\documentclass[12pt,twoside]{article}
\usepackage{amsfonts}
\usepackage{cite}
 \usepackage{color}
\usepackage{amsmath} \numberwithin{equation}{section}
\usepackage{amsthm}
\topmargin by-\headsep  \textheight 24cm
\newtheorem{lemma}{Lemma}[section]

\newtheorem{definition}{Definition}[section]
\newtheorem{theorem}[definition]{Theorem}

\newtheorem{remark}{Remark}[section]

\begin{document}
\title{Vanishing shear viscosity  and boundary layers for  plane magnetohydrodynamics flows}
\author{Wenshu Zhou$^1$\thanks{Corresponding author. \newline E-mail address:   mathpde@163.com (W.Zhou), qin\_xulong@163.com(X. Qin); mathqcy@163.com (C. Qu) } \quad Xulong Qin$^2$\quad Chengyuan Qu$^1$\\
 \small 1. Department of Mathematics, Dalian Nationalities University,  Dalian 116600, China\\
 \small  2. Department of Mathematics, Sun Yat-sen University, Guangzhou 510275,  China\\}


\date{}
 \maketitle

\begin{abstract}
\small{In this paper, we consider an initial-boundary problem for  plane magnetohydrodynamics
flows under the general condition on the heat conductivity $\kappa$
that may depend on both the density $\rho$ and the temperature $\theta$ and satisfies
$$
\kappa(\rho,\theta)\geq\kappa_1(1+\theta^{q}) \quad \hbox{\rm with
constants}~ \kappa_1>0 ~\hbox{\rm and}~ q>0.
$$
We prove the global existence of strong solutions for large initial
data  and justify the passage to the limit as the shear viscosity
$\mu$ goes to zero.   Furthermore, the value $\mu^\alpha$ with any $0<\alpha<1/2$ is established for the boundary layer  thickness. }


\end{abstract}

%
\bigbreak

\section{Introduction}Magnetohydrodynamics (MHD) concerns the motion of conducting fluids in
an electromagnetic field and has a very broad range of applications.
The dynamic motion of the fluid and the magnetic field interact
strongly on each other, so the hydrodynamic and electrodynamic
effects are coupled, which make the problem considerably
complicated.
The plane MHD flows  with constant longitudinal magnetic field, which are   three-dimensional MHD flows
  uniform in the transverse
direction, are governed by the following equations:

 \begin{equation}\label{e1}
  \begin{split}
 &\rho_t+(\rho u)_x=0,\\
 &(\rho u)_t+\left(\rho u^2+p+\frac12 |\mathbf{b}|^2\right)_x
 =(\lambda u_{x})_x,\\[1mm]
 &(\rho \mathbf{w})_t+(\rho u \mathbf{w}-\mathbf{b})_x=(\mu\mathbf{w}_{x})_x,\\[1mm]
 &\mathbf{b}_t+(u\mathbf{b}-\mathbf{w})_x=(\nu\mathbf{b}_{x})_x,\\[1mm]
& (\rho e)_t+(\rho  u e)_x-(\kappa\theta_x)_x+pu_x={\cal Q},\\[1mm]
&{\cal Q}:=\lambda u_x^2+\mu
|\mathbf{w}_x|^2+\nu |\mathbf{b}_x|^2,
 \end{split}
 \end{equation}
  where  $\rho$ denotes the density of the flow, $\theta$  the temperature,  $u\in \mathbb{R}$
  the longitudinal velocity, $\mathbf{w}=(w_1,w_2)\in \mathbb{R}^2$ transverse velocity,
  $\mathbf{b}=(b_1,b_2)\in \mathbb{R}^2$ transverse magnetic field, $p=p(\rho,\theta)$ the pressure,
  $e=e(\rho,\theta)$ the internal energy, and $\kappa=\kappa(\rho,\theta)$   the heat conductivity.
 The coefficients $\lambda$, $\mu$ and $\nu$ standing for the bulk viscosity, shear viscosity and
 the magnetic diffusivity, respectively,  are assumed to be positive constants in this paper.
 We  focus on the perfect gas
 with the equations of state:
 \begin{equation}\label{e2}
 p=\gamma\rho \theta, \qquad e=c_v\theta,
 \end{equation}
where the constants $\gamma>0$ and $c_v>0$. Without loss of
generality, we set $c_v=1$.

We consider system \eqref{e1} in the bounded domain $Q_T=\Omega\times (0, T)$ with $\Omega=(0, 1)$
subject to the following initial and boundary conditions:
\begin{equation}\label{e4}
\left\{
\begin{array}{lllllll}
(\rho,u,\theta, \mathbf{w},\mathbf{b})(x,0)=(\rho_0,u_0,\theta_0,\mathbf{w}_0,\mathbf{b}_0)(x),\\[1mm]
(u, \mathbf{b},\theta_x)|_{x=0, 1}=\mathbf{0},\\[1mm]
\mathbf{w}(0,t)=\mathbf{w}^-(t),\quad\mathbf{w}(1,t)=\mathbf{w}^+(t).
\end{array}
\right.
\end{equation}

Because of physical importance, complexity, rich phenomenon, and
mathematical challenges, the MHD problem  has been extensively
studied in many papers, see
\cite{B,CW1,CW2,Wang,CD,DF,HW,HW2,VH,KO,LL,LZ,S0,W} and the references
therein.   In particular, if there is no  magnetic effect, MHD
reduces to the compressible Navier-Stokes equations, see for example
\cite{K,F,S,TW,WX,XY} and references therein for some mathematical
studies. However, many fundamental problems for MHD are still open.
For example,  even though for the one dimensional
 compressible Navier-Stokes equations, there is a pioneer work  by
Kazhikhov and Shelukhin \cite{KS} on the global existence of strong
solutions with large initial data, the corresponding problem for the
MHD system with constant viscosity, heat conductivity and
diffusivity coefficients remains unsolved. The reason is that the
presence of the magnetic field and complex interaction with the
hydrodynamic motion in the MHD flow of large oscillation cause
serious difficulties.

The  initial-boundary value problem  \eqref{e1}-\eqref{e4} has
fundamental importance in the studies on the MHD problem. In this
paper, we investigate the global existence, zero shear viscosity
limit, convergence rate and boundary layer effect
 of strong solutions for problem \eqref{e1}-\eqref{e4} with large initial data, where $\kappa$ may depend both density
and temperature such that $\kappa=\kappa(\rho,\theta)$ is twice
continuous differential in $\mathbb{R}^+\times\mathbb{R}^+$ and
satisfies
\begin{equation}\label{kappa}
\kappa(\rho,\theta)\geq \kappa_1(1+\theta^q)\quad\hbox{\rm with
 constants}~ \kappa_1>0 ~\hbox{\rm and}~ q>0.
\end{equation}
 In kinetic theory of gas, the heat conductivity $\kappa$
is a function of temperature $\theta$ and increases with $\theta$ in
general (cf.\cite{CC,ZYY}). From experimental results for gases at
very high temperatures (see \cite{ZYY}),   the condition
\eqref{kappa} seems reasonable  when one considers a gas model that
incorporates real gas effects that occur in high temperature (cf.\cite{K}). In \cite{K}, one of the assumptions on $\kappa$ is that
there are constants $C_1, C_2>0$ such that the heat conductivity
$\kappa$ satisfies
\begin{equation}\label{assumption30}
\begin{split}
 C_1(1+\theta^q)\leq \kappa(\rho,\theta)\leq C_2(1+\theta^q), \quad\forall \rho, \theta>0,\\
\end{split}
\end{equation}
where $q\geq 2$, which implies that $\kappa$ has a positive lower
bound. This type of temperature dependence is also motivated by the
physical fact that $\kappa$ grows like $\theta^q $ with   $q=2.5$ for important physical regimes and $q \in [4.5,
5.5]$ for molecular diffusion in gas (see \cite{ZYY}).  The
assumption \eqref{assumption30} with $q> 0$ was also made  in many papers (see for example \cite{Wang,CW2,FanJiang,
FHL, JZ} and references  therein). Clearly, here we
remove these assumptions on $\kappa$.

 The well-posedness theory has been studied in many papers, some of which will be mentioned below.
 It was Vol'pert and Hudjaev \cite{VH} who first proved the existence and uniqueness of local smooth solutions.
 The global existence of smooth solutions with small initial data was established by Kawashima and Okada \cite{KO}.
  Under the technical condition on $\kappa$:
 \begin{equation}\label{assumption4}
\begin{split}
 C^{-1}(1+\theta^q)\leq \kappa(\rho,\theta)\equiv\kappa(\theta)\leq C (1+\theta^q),
\end{split}
\end{equation}
 for  $q\geq 2$, Chen and Wang \cite{CW2} proved the existence, uniqueness and the Lipschitz
continuous dependence of global strong solutions with large $H^1$ initial
data. Similar results can be found in \cite{CW1, Wang} under the same technical condition as \eqref{assumption4}.
The existence of global  weak solutions was proved by Fan, Jiang and Nakamura \cite{FanJiang}
under the condition \eqref{assumption4} for $q\geq 1$ or the condition $\kappa\equiv\kappa(\rho)\geq C/\rho$,
 while the uniqueness and the Lipschitz continuous
dependence on the initial data of global weak solutions with the initial data  in  Lebesgue spaces were
obtained by them in \cite{FanJiang11}. Very recently, the case $q>0$ of condition
 \eqref{assumption4} was treated by Fan, Huang and Li \cite{FHL} where the existence and
 uniqueness of global solutions with large initial data and vaccum  were shown. A similar result  can be found in \cite{HJ} by Hu and Ju. The uniqueness and continuous dependence of weak
solutions for the Cauchy problem have been proved recently by Hoff and Tsyganov in
\cite{HT}. In this paper, we   show the global existence of strong solutions for problem \eqref{e1}-\eqref{e4}
under the general condition \eqref{kappa}, which extends some global existence results mentioned above.

 The problem of small viscosity finds many applications, for example, in the boundary layer theory (cf. \cite{Sch}).
 In this direction, some results on the Navier-Stokes equations can be referred to  \cite{S,JZ, FS,FS1,QYYZ,YaoZhangZhu}
  and references   therein. The vanishing shear viscosity limit of the weak solution for problem \eqref{e1}-\eqref{e4}
   was studied by Fan, Jiang and Nakamura  \cite{FanJiang}
under the condition \eqref{assumption4} for $q\geq 1$ or $\kappa\equiv \kappa(\rho)\geq C/\rho$.
As pointed out in  \cite{FHL}, the result of \cite{FanJiang}   can be  transplanted to the case $q>0$ of
the condition \eqref{assumption4}. In this paper, we justify the passage to the limit
with more strong convergence of $\mathbf{w}$ and $\mathbf{b}$ under the general condition \eqref{kappa}.
Thus, we extend and improve some results mentioned above.

 The boundary layer theory has been one of the fundamental
and important issues in fluid dynamics since it was proposed by
Prandtl in 1904. Frid and Shelukhin  in \cite{FS1}  investigated the
boundary layer effect of the compressible isentropic Navier-Stokes
equations  with cylindrical symmetry, and proved the existence of
boundary layers of   thickness   $O(\mu^{\alpha}) (0<\alpha<1/2)$.
Under the assumption on $\kappa$:
\begin{equation}\label{assumption3}
\begin{split}
 C^{-1}(1+\theta^q)\leq \kappa(\rho,\theta)\leq C (1+\theta^q),\quad
 |\kappa_\rho(\rho,\theta)|\leq C (1+\theta^q), ~~q>1, \\
\end{split}
\end{equation}
 Jiang and Zhang  \cite{JZ} studied the compressible  nonisentropic Navier-Stokes equations  with cylindrical symmetry,
 and proved that the thickness of boundary layer  is of the order $O(\mu^\alpha)
 (0<\alpha<1/2)$. Recently, Jiang and Zhang's result is  extended to the case of
 constant heat conductivity, see \cite{QYYZ}. A similar result  can be found in \cite{YaoZhangZhu} by Yao,
Zhang and Zhu.
 To the best of our knowledge, however, there is no corresponding results   for the initial
boundary problem \eqref{e1}--\eqref{e4}.  In this paper, the value
$\mu^\alpha (0<\alpha<1/2)$ is established for the boundary layer
thickness of problem \eqref{e1}-\eqref{e4}.

 We introduce some
notations. Let $k\geq 0$ be an integer, $\mathcal{O}$ a domain of $\mathbb{R}^n (n\geq 1)$ and $p\geq 1$.  $W^{k,p}(\mathcal{O})$ and $W_0^{k,p}(\mathcal{O})$  denote the usual Sobolev spaces, $W^{0,p}(\mathcal{O})=L^p(\mathcal{O})$.
$C^{k}(\mathcal{O})$ and $C^{k}(\overline{\mathcal{O}})$ denote  the spaces consisting of continuous
derivatives up to order $k$ in $\mathcal{O}$. For $0<\alpha<1$, $C^{k+\alpha}(\mathcal{O})$ (resp.  $C^{\alpha}(\overline{\mathcal{O}})$) and  $C^{k+\alpha,k+\alpha/2}(\mathcal{O})$ (resp. $C^{k+\alpha,k+\alpha/2}(\overline{\mathcal{O}})$) denote  the H\"{o}lder spaces with the exponent $\alpha$.  $L^p(I,B)$ is the
space of all strong measurable, $p^{th}$-power
integrable (essentially bounded if $p=\infty$) functions from $I$ to
$B$, where $I\subset \mathbb{R}$ and $B$ is a Banach space. For
simplicity, we also use the notation $\|(f, g,
\cdots)\|^2_{B}=\|f\|^2_{B}+\|g\|_{B}^2+\cdots$ for functions $f,
g,\cdots$ belonging to
  $B$ equipped with a norm $\|\cdot\|_{B}$.

In what follows, we assume that the initial and boundary functions satisfy
 \begin{equation}\label{assumption1}
\begin{split}
&\rho_0>0,\,\,\theta_0>0,\,\,\, \|(\rho_0^{-1},\theta_0^{-1})\|_{C(\overline{\Omega})}<\infty,
\,\,\|(\mathbf{w}^-,\mathbf{w}^+)\|_{C^1([0,T])}<\infty, \\[1mm]
&(\rho_0,\mathbf{w}_0,\theta_0)\in W^{1,2}(\Omega),\,\,\mathbf{b}_0\in W_0^{1,2}(\Omega),
\,\, u_0\in W_0^{1,2}(\Omega)\cap W^{2,m}(\Omega)~\hbox{\rm with}~m\in (1, +\infty),\\[1mm]
&\mathbf{w}_0(0)=\mathbf{w}^-(0),\,\,\mathbf{w}_0(1)=\mathbf{w}^+(0).
\end{split}
\end{equation}
Now the  results on the global existence, vanishing shear viscosity limit and convergence rate
of strong solutions can be stated as
follows.

\begin{theorem}\label{existencethm}
Let \eqref{kappa} and \eqref{assumption1} hold. Then

(i)~~ For any fixed $\mu>0$, there exists a unique strong solution
$(\rho,u,\mathbf{w},\mathbf{b},\theta)$ for problem
\eqref{e1}--\eqref{e4}. Moreover, there exist  some positive
constants $C$ independent of $\mu$ such that
\begin{equation}\label{ve}
\begin{split}
&C^{-1}\leq \rho, \theta \leq C,\quad  \|(u,\mathbf{w},\mathbf{b})\|_{L^\infty(Q_T)}\leq C,\\[2mm]
&\|(\rho_t,\rho_x, u_x,\mathbf{b}_x,\theta_x)\|_{L^\infty(0, T;L^2(\Omega))}
+\|(u_t, \mathbf{b}_t,\theta_t, u_{xx}, \theta_{xx})\|_{L^2(Q_T)}\leq C,\\[2mm]
&\|\mathbf{w}_x\|_{L^\infty(0,
T;L^1(\Omega))}+\|\mathbf{w}_t\|_{L^2(Q_T)}\leq C,\\
&\mu^{1/4}\|\mathbf{w}_x\|_{L^{\infty}(0,T;L^2(\Omega))} +\mu^{3/4}\|\mathbf{w}_{xx}\|_{L^2(Q_T)} \leq C,\\
&\|\sqrt{\omega}\mathbf{w}_x\|_{L^\infty(0,
T;L^2(\Omega))}+\|\big(\sqrt{|u|}\mathbf{w}_x,
\sqrt{\omega}\mathbf{b}_{xx}\big)\|_{L^2(Q_T)} \leq C,\\
\end{split}
\end{equation}
where $\omega: [0, 1]\rightarrow [0, 1]$ is defined by
\begin{equation*}
\begin{aligned}
\omega(x)=\left\{\begin{aligned} &x,&0\leq x\leq  1/2,\\ &1-x,&
1/2\leq x\leq 1.
\end{aligned}\right.
\end{aligned}
\end{equation*}

(ii)~~There exist  some functions $\overline\rho, \overline u,
\overline{\mathbf{w}}, \overline{\mathbf{b}}$ and $\overline\theta$ in the class:
\begin{equation*}\label{base0}
\mathbb{F}:\left\{\begin{split}
&  \overline\rho, \overline\theta>0,\quad(\overline u, \overline{\mathbf{b}})|_{x=0, 1}=0,\\
&(\overline\rho,  1/\overline\rho,\overline u, \overline{\mathbf{w}}, \overline{\mathbf{b}},
\overline\theta,1/\overline\theta)\in L^\infty(Q_T),\quad
  \overline{\mathbf{w}} \in L^\infty(0, T;W^{1,1}(\Omega))\cap BV(Q_T),\\
    &(\overline\rho_t,\overline\rho_x,\overline u_x,\overline{\mathbf{b}}_x,\overline\theta_x)
    \in L^\infty(0, T;L^2(\Omega)),\quad (\overline u_x,\overline{\mathbf{b}}_x,\overline \theta_x) \in L^2(0, T;L^\infty(\Omega)),\\
   &\big(\overline u_t,  \overline{\mathbf{w}}_t, \overline{\mathbf{b}}_t,\overline\theta_t,
   \overline u_{xx}, \overline\theta_{xx}\big) \in L^2(Q_T),\\
       &\sqrt{\omega}\overline{\mathbf{w}}_x\in L^\infty(0, T;L^2(\Omega)),\quad\big(\sqrt{|\overline u|}\overline{\mathbf{w}}_x,\sqrt{\omega}\overline{\mathbf{b}}_{xx} \big) \in  L^2(Q_T),\\
 \end{split}\right.
\end{equation*}
such that, as $\mu\rightarrow 0$,  $(\rho,u,\mathbf{w},\mathbf{b}, \theta)$   converges in the following sense
\begin{equation*}\label{rate}
 \begin{split}
 &(\rho,u,\mathbf{b}, \theta)\rightarrow (\overline\rho,\overline u,\overline{\mathbf{b}},
 \overline\theta)~~ \hbox{\rm strongly in}~~C^\alpha(\overline Q_T),~\forall \alpha\in(0, 1/4),\\
 &(u_x, \theta_x)\rightarrow (\overline u_x, \overline\theta_x)~~\hbox{\rm strongly in}
 ~~L^{s_1}(Q_T),~\forall s_1 \in [1, 6),\\
 & \mathbf{b}_x \rightarrow  \overline{\mathbf{b}}_x ~~\hbox{\rm strongly in}~~L^{s_2}(Q_T),~\forall s_2 \in [1, 4),\\
    &  (\rho_t,\rho_x) \rightharpoonup (\overline\rho_t, \overline\rho_x)
    ~~\hbox{\rm weakly}-*~\hbox{\rm in}~ L^\infty(0, T; L^2(\Omega)),\\
  &(u_t,\mathbf{b}_t,\theta_t,u_{xx}, \theta_{xx})
  \rightharpoonup(\overline u_t,\overline{\mathbf{b}}_t, \overline\theta_t,\overline u_{xx},
  \overline\theta_{xx})~~ \hbox{\rm weakly in}~~L^2(Q_T),\\
  & \mathbf{b}_{xx}  \rightharpoonup  \overline{\mathbf{b}}_{xx}
  \quad \hbox{\rm  weakly in}~~L^2\big((a+\delta,b-\delta)\times(0, T)\big),~\forall \delta\in\big(0, (b-a)/2\big),\\
      \end{split}
  \end{equation*}
 and
  \begin{equation*}
   \begin{split}
 & \mathbf{w}  \rightarrow   \overline{\mathbf{w}}\quad\hbox{\rm strongly in}~~C^\alpha([a+\delta,b-\delta]\times[0, T]),~\forall \delta\in\big(0, (b-a)/2\big),~\alpha\in(0, 1/4)\\
 &  \mathbf{w}_t   \rightharpoonup   \overline{\mathbf{w}}_t ~~\hbox{\rm weakly in}~~L^2(Q_T),\\
  &   \mathbf{w}_x \rightharpoonup  \overline{\mathbf{w}}_x\quad\hbox{\rm weakly}-*~\hbox{\rm in}~ L^\infty(0, T; L^2(a+\delta,b-\delta)),~\forall \delta\in\big(0, (b-a)/2\big),\\
 &\mathbf{w}\rightarrow  \overline{\mathbf{w}}~~ \hbox{\rm strongly in}~~L^r(Q_T),\quad\forall r \in [1, +\infty),\\
 &\sqrt{\mu}\|\mathbf{w}_x\|_{L^4(Q_T)} \rightarrow 0.
   \end{split}
    \end{equation*}
  Moreover, $(\overline{\rho},\overline{u},\overline{\mathbf{w}},\overline{\mathbf{b}},\overline{\theta})$
  solves problem  \eqref{e1}--\eqref{e4} with $\mu=0$ in the sense:
  \begin{equation}\label{equations}
 \begin{split}&\left.\begin{split}
&\overline\rho_t+(\overline\rho ~\overline u)_x=0, \\[1mm]
&(\bar\rho\bar u)_t+\big(\bar\rho\bar u^2+\gamma\overline\rho\overline\theta+|\overline{\mathbf{b}}|^2/2\big)_x
=\lambda \overline u_{xx},\\[1mm]
&(\bar\rho\overline{\mathbf{w}})_t+ (\bar\rho\bar u\overline{\mathbf{w}}-\overline{\mathbf{b}})_x=0,\\[1mm]
& \overline{\mathbf{b}}_t+ ( \bar u\overline{\mathbf{b}}-\overline{\mathbf{w}})_x=\nu \overline{\mathbf{b}}_{xx},\\[1mm]
&   (\overline\rho \overline\theta)_t+(\bar\rho \bar u\overline\theta)_x +\gamma\overline\rho
\overline\theta \overline u_x
-\big[\kappa(\overline\rho,\overline\theta)\overline\theta_x\big]_x
= \lambda \overline u_x^2+\nu|\overline{\mathbf{b}}_x|^2,
\end{split}\right\}~\hbox{\rm a.e. in }~Q_T,\\[1mm]
 &\iint_{Q_T} \Big\{\big[(\overline\rho \overline\theta)_t+(\bar\rho\bar u\overline\theta)_x
+\gamma\overline\rho \overline\theta \overline u_x-\lambda \overline u_x^2-\nu|\overline{\mathbf{b}}_x|^2\big]\varphi
+\kappa(\overline\rho,\overline\theta)\overline\theta_x\varphi_x\Big\}dxdt
= 0,
\end{split}\end{equation}
for all  $\varphi \in L^2(0, T;W^{1,2}(\Omega))$.

(iii)~~Assume that
$(\overline{\rho},\overline{u},\overline{\mathbf{w}},\overline{\mathbf{b}},\overline{\theta})
\in \mathbb{F}$ is a solution for the limit problem
\eqref{equations}. Then
\begin{equation*}\label{0u3}
\begin{aligned}
&\|(\rho-\overline\rho, u-\overline u, \mathbf{w}-\overline{\mathbf{w}},\mathbf{b}-\overline{\mathbf{b}},
\theta-\overline\theta )\|_{L^\infty(0, T;L^2(\Omega))}\\
&\quad\quad\quad+\|(u_x-\overline u_x,  \mathbf{b}_x-\overline{\mathbf{b}}_x, \theta_x-\overline\theta_x)\|_{L^2(Q_T)}=
 O(\mu^{1/4}).
\end{aligned}
\end{equation*}

\end{theorem}
\begin{remark}
With the estimates appearing in   the above
theorem, and following the argument given in \cite{KS}(cf. \cite{CW1}), if the initial data is in H\"{o}lder space,
 i.e.,
\begin{equation*}
\rho_0\in C^{1+\alpha}(\Omega),\,\,\quad (u_0,\mathbf{w}_0,\mathbf{b}_0,\theta_0)\in C^{2+\alpha}(\Omega)
\end{equation*}
for some $\alpha\in (0,1)$, then there exists a unique classical solution
\begin{equation*}
\rho\in C^{1+\alpha,1+\alpha/2}(Q_T),\,\,\quad (u,\mathbf{w},\mathbf{b},\theta)\in C^{2+\alpha,1+\alpha/2}(Q_T),
\end{equation*}
and it satisfies \eqref{ve}.
\end{remark}

\begin{remark}
It should be pointed out that if we only consider the global existence with  fixed $\mu$, then
 the condition $u_0\in W^{2,m} (m>1)$ in \eqref{assumption1} can be removed. In fact, this can be done
 in a more easy way, but we do not pursue it in the paper.
\end{remark}

 Compared to \cite{FanJiang,FHL} and some related references, the generality of the condition \eqref{kappa} causes
 some other technical difficulties  since all the
estimates in \eqref{ve} must be uniform in $\mu$. Firstly,  we must
overcome the difficulty coming from the dissipative estimate on the
temperature. For example, Fan-Jiang-Nakamura \cite{FanJiang} only
 established the $\mu$-uniform estimate of $\theta_x$ in $L^\beta(Q_T)$ with any $ \beta\in(1, 3/2)$
 by means of the technique used by Frid and Shelukhin \cite{FS}.   Secondly,  to
obtain the stronger convergence  of $\mathbf{w}$ and $\mathbf{b}$
(see Theorem 1.1(ii)), we must establish some new  uniform estimates
on the derivatives of $\mathbf{w}$  and $\mathbf{b}$. Thirdly, we
must seek a new method to obtain a uniform upper bound of the
temperature.

To overcome the difficulties, some  techniques are developed
here. One of two  ingredients in the proof is the boundary estimates
of derivatives of the transverse velocity and the magnetic field,
and the other is that we deduce a  uniform upper bound of $\theta$ by
a simple, direct method.

Below we present a sketch of the proof to \eqref{ve}. Firstly, the
uniform upper and lower bounds of the density can be obtained in a
standard way. Next, a key observation is that
we can establish the uniform bound of
$\|u_{xx}\|_{L^{m_0}(Q_T)} (m_0>1)$ by $L^p$-theory of
linear parabolic equations (see Lemma \ref{2.5}), which plays an
important role in this paper. It should be pointed out that it is in
this step we ask the condition $u_0\in W^{2,m}(\Omega)$ for some
$m>1$. By virtue of the estimate and a delicate analysis, we then
deduce the difficult bounds of $\|\omega\mathbf{w}_x\|_{L^\infty(0,
T;L^2(\Omega))}$ and
 $\|(u_t,\mathbf{b}_t,\mathbf{w}_t,
u_{xx},\theta_{x},\omega\mathbf{b}_{xx})\|_{L^{2}(Q_T)}$(see Lemma \ref{2.10}).
In this step, the main idea is to use the norm $\|u_{xx}\|_{L^{2}(Q_T)}$ to control the qualities
 $\|(\omega\mathbf{w}_x, \mathbf{b}_x)\|_{L^\infty(0, T;L^2(\Omega))}$ and
 $\|\mathbf{w}_t\|_{L^2(Q_T)}$ (see Lemmas \ref{2.7}-\ref{2.9}) and then,
 from the equations of $u$ and $\theta$ it follows the uniform bound of $\|u_{xx}\|_{L^{2}(Q_T)}$
 by Gronwall's inequality.
 With the  uniform bound  of $\|\mathbf{b}_t\|_{L^{2}(Q_T)}$,
 we deduce the uniform bounds of $\|\mathbf{w}_x\|_{L^\infty(0, T;L^1(\Omega))}$
 and $\|\mathbf{b}_x\|_{L^2(0, T;L^\infty(\Omega)}$ (see Lemmas \ref{2.11} and \ref{2.12}),
 by which  we   further obtain  the uniform bounds of $\|\sqrt{\omega}\mathbf{w}_x\|_{L^\infty(0, T;L^2(\Omega))}$
 and $\big(\mu^{1/4}\|\mathbf{w}_x\|_{L^\infty(0,T;L^2(\Omega))}+\mu^{3/4}\|\mathbf{w}_{xx}\|_{L^2(Q_T)}\big)$
 (see Lemma \ref{2.13}), which are essential to study both  $L^2$  convergence rate and boundary layer thickness.
 Due to the above   estimates, we finally get an upper bound of
$\theta$ in a direct way (see Lemma \ref{2.14}). As a
consequence, the uniform bound of $\|(\theta_t,
\theta_{xx})\|_{L^{2}(Q_T)}$ can be obtained by a brief argument
(see Lemma \ref{2.15}). Consequently, the passage to limit is
justified in the more strong sense.

Next, we  investigate the thickness of boundary layer.  At first, we
  give the definition of a BL-thickness  defined as in
\cite{FS1} (cf. \cite{JZ}).
\begin{definition}\label{defination}
A function $\delta(\mu)$ is called a BL-thickness for problem
\eqref{e1}-\eqref{e4} with vanishing  $\mu$ if
$\delta(\mu)\downarrow 0$ as $ \mu \downarrow 0$, and
\begin{equation*}\label{12}
\begin{aligned}
&\lim\limits_{\mu\rightarrow 0}\|(\rho-\overline\rho, u-\overline u,
\mathbf{w}-\overline{\mathbf{w}}, \mathbf{b}-\overline{\mathbf{b}},
\theta-\overline\theta)\|_{L^\infty(0,T;L^\infty(\Omega_{\delta(\mu)}))}=0,\\
&\mathop{ \inf\lim}\limits_{\mu\rightarrow 0}\|(\rho-\overline\rho,
u-\overline u, \mathbf{w}-\overline{\mathbf{w}},
\mathbf{b}-\overline{\mathbf{b}},\theta-\overline\theta)\|_{L^\infty(0,T;L^\infty(\Omega))}>0,
\end{aligned}
\end{equation*}
where $\Omega_\delta=(\delta, 1-\delta)$ for $\delta \in (0, 1/2)$,
and $(\rho, u, \mathbf{w}, \mathbf{b}, \theta)$ and
$(\overline\rho,\overline u, \overline{\mathbf{w}},
\overline{\mathbf{b}}, \overline\theta)$ are the solutions to
problem \eqref{e1}-\eqref{e4}  and problem \eqref{e1}-\eqref{e4}
with $\mu=0$, respectively.
\end{definition}

 We shall prove that for any $\alpha \in (0, 1/2)$, the function
 $\delta(\mu)=\mu^{\alpha}$
is a BL-thickness, which is almost optimal since it is close to the
classical value $O(\sqrt{\mu})$ (see e.g. \cite{Sch}). One can see from
the proof in Section 3 that our method is a bit different from that
used in \cite{JZ,FS1}, which is based on an iteration inequality \eqref{iteration}. To indicate the
 idea clearly, we further assume that
\begin{equation}\label{vw}
\begin{aligned}
\mathbf{w}_0=\mathbf{b}_0\equiv\mathbf{0}.
\end{aligned}
\end{equation}

\begin{theorem}\label{am} Let  \eqref{kappa},
\eqref{assumption1} and \eqref{vw} hold. Assume that $(\mathbf{w}^-, \mathbf{w}^+)$ is not identically equal to
$\mathbf{0}$. Then the limit problem \eqref{equations} has a unique solution $(\overline \rho,\overline u, \mathbf{0},
 \mathbf{0}, \overline\theta)$ in $\mathbb{F}$, and the  function $\delta(\mu)=\mu^\alpha$
  for any $\alpha\in(0, 1/2)$
is  a BL-thickness for problem \eqref{e1}-\eqref{e4} such that
\begin{equation*}
\begin{split}
   &\lim\limits_{\mu\rightarrow 0} \|(\rho-\overline\rho,u-\overline u,\mathbf{b},\theta-\overline\theta)\|_{C^\alpha(\overline Q_T)}=0,\quad \forall \alpha\in (0,1/4),\\
   &\lim\limits_{\mu\rightarrow 0} \|\mathbf{w}\|_{L^\infty(0, T;L^\infty(\delta(\mu),
   1-\delta(\mu)))}=0,\quad \mathop{\inf\lim}\limits_{\mu\rightarrow 0} \|\mathbf{w}\|_{L^\infty(0, T;L^\infty(\Omega))}>0.\\
  \end{split}
\end{equation*}
Moreover, $\mathbf{w}$ has the asymptotic property: 
\begin{equation*}
\begin{split}
      \|\mathbf{w}_x\|^2_{L^\infty(0, T;L^2(\delta, 1-\delta))}\leq \left\{
              \begin{split}
           & C_n\big(\tau+\tau^3+\cdots+\tau^{n-2}\big)+C_n\mu^{(n-1)/2}/\delta^n~(n=\hbox{\rm odd}),\\
       & C_n\big(\tau+\tau^3+\cdots+\tau^{n-1}\big)+C_n\mu^{(n-1)/2}/\delta^n~(n=\hbox{\rm even}),\\
         \end{split}\right.
  \end{split}
\end{equation*}
where  $\delta\in (0, 1/2), \tau=\sqrt{\mu}/\delta$, and the constants $C_n$ are independent of $\mu$ and $\delta$.
\end{theorem}

The remainder of this paper shall be arranged as follows. In Section
2, we will prove Theorem \ref{existencethm}. For this, a lot of a
priori estimates independent of $\mu$ are derived in Section 2.1,
which are sufficient to prove  this theorem.  The second and third
parts of this theorem can be shown in Sections 2.2 and 2.3,
respectively. Finally, we will give the proof of Theorem \ref{am} in
Section 3.
\section{The proof of Theorem 1.1}

The existence and uniqueness of local solutions can be obtained by
using the Banach theorem and the contractivity of the operator
defined by the linearization of the problem on a small time interval
(cf.\cite{Wang,Nash}). The existence of global solutions is proved
by extending the local solutions globally in time based on the
global a priori estimates of solutions. The uniqueness of the global
solution follows from the uniqueness of the local solution. Thus,
the next subsection will focus on deriving required a priori
estimates of the  solution  $(\rho,u,\mathbf{w}, \mathbf{b},\theta)$. Moreover, all a
priori estimates which will be established are uniform in $\mu$.

Throughout this section, we shall denote by $C$ the various positive
constants dependent on $T$, but independent of $\mu$.

\subsection{A priori estimates  independent of $\mu$}

Rewrite  \eqref{e1} as
\begin{equation}\label{e20}
\begin{split}
&\mathcal{E}_t+\Big[u\big(\mathcal{E}+p+\frac12|\mathbf{b}|^2\big)-\mathbf{w}\cdot\mathbf{b}\Big]_x=\big(\lambda uu_x+\mu\mathbf{w}\cdot\mathbf{w}_x+\nu\mathbf{b}\cdot\mathbf{b}_x+\kappa\theta_x\big)_x,\\[2mm]
&(\rho \mathcal{S})_t+(\rho u \mathcal{S})_x-\left(\frac{\kappa\theta_x}{\theta}\right)_x=\frac{\lambda u_x^2+\mu|\mathbf{w}_x|^2+\nu|\mathbf{b}_x|^2}{\theta}
+\frac{\kappa\theta_x^2}{\theta^2},
\end{split}
\end{equation}
where $\mathcal{E}$ and $\mathcal{S}$ are the total energy and the
entropy, respectively,
\begin{equation*}
\begin{split}
&\mathcal{E}=\rho\left[\theta+\frac12(u^2+|\mathbf{w}|^2)\right]+\frac12|\mathbf{b}|^2,\quad\mathcal{S}
=\ln\theta-\gamma\ln\rho.
\end{split}
\end{equation*}

\begin{lemma}\label{2.1}
\label{energy}Let \eqref{kappa} and \eqref{assumption1} hold. Then
\begin{equation}\label{ba1}
\begin{split}
&\int_\Omega \rho(x,t)dx=\int_\Omega \rho_0(x)dx,\quad\forall t \in (0, T),\\
&\sup\limits_{0<t<T}\int_\Omega\big[\rho(\theta+ u^2+|\mathbf{w}|^2)+ |\mathbf{b}|^2\big] dx\leq
C,\\[1mm]
& \iint_{Q_T}
\left(\frac{\lambda u_x^2+\mu|\mathbf{w}_x|^2+\nu|\mathbf{b}_x|^2}{\theta}+\frac{\kappa\theta_x^2}{\theta^2}\right)dxdt
\leq C.
\end{split}
\end{equation}
\end{lemma}
\begin{proof}
Integrating $\eqref{e20}_1$ over $Q_t=\Omega\times(0, t)$ with $t\in (0, T)$ yields
\begin{equation}\label{total}
\begin{split}
&\int_\Omega \mathcal{E}dx=\int_\Omega \mathcal{E}|_{t=0}dx+\mu\int_0^t\mathbf{w}\cdot\mathbf{w}_x|_{x=0}^{x=1}ds.
\end{split}
\end{equation}

Let $a=0$ or $1$. We first integrate  \eqref{e1}$_3$ from $x=a$ to $x$, and then integrate the resulting equation over $\Omega$, to obtain
\begin{equation*}\label{w8}
\begin{split}
\mu\mathbf{w}_x(a,t)=\mu\big(\mathbf{w}^+- \mathbf{w}^-\big)-\int_\Omega(\rho u\mathbf{w}-\mathbf{b})dx-\frac{\partial}{\partial t}\int_\Omega\int_a^x\rho\mathbf{w}dydx.
\end{split}
\end{equation*}
Multiplying it by  $\mathbf{w}(a,t)$ and integrating  over $(0, t)$, we have
\begin{equation*}
\begin{split}
\mu\int_0^t(\mathbf{w}\cdot\mathbf{w}_x)(a,s)ds=&\mu\int_0^t\big(\mathbf{w}^+- \mathbf{w}^-\big)\cdot\mathbf{w}(a,s)ds -\int_0^t\mathbf{w}(a,s)\cdot\left(\int_\Omega(\rho u\mathbf{w}-\mathbf{b})dx\right)ds\\[1mm]
&-\mathbf{w}(a,t)\cdot\left(\int_\Omega\int_a^x
\rho\mathbf{w}dydx\right)+\mathbf{w}(a,0)\cdot\left(\int_\Omega\int_a^x
\rho_0\mathbf{w}_0dydx\right)\\[1mm]
&+\int_0^t\mathbf{w}_t(a,t)\cdot\left(\int_\Omega\int_a^x\rho\mathbf{w}dydx\right)dt,
\end{split}
\end{equation*}
hence,   by Young's inequality and \eqref{ba1}$_1$,
\begin{equation*}\label{w10}
\begin{split}
\left|\mu\int_0^t(\mathbf{w}\cdot\mathbf{w}_x)(a,s)ds\right|\leq&C+C\int_\Omega\rho|\mathbf{w}|dx+C \iint_{Q_t}\big(\rho |u| |\mathbf{w}|+|\mathbf{b}|+\rho  |\mathbf{w}|\big)dxds\\
\leq&C+\frac12\int_\Omega\mathcal{E}dx+C\iint_{Q_t}\mathcal{E}dxds.
\end{split}
\end{equation*}
Substituting it into \eqref{total} yields
\begin{equation*}
\begin{split}
&\int_\Omega \mathcal{E}dx\leq C+C\iint_{Q_t}\mathcal{E}dxds,
\end{split}
\end{equation*}
and so,  \eqref{ba1}$_2$ follows from Gronwall's inequality.

\eqref{ba1}$_3$ can be proved by integrating \eqref{e20}$_2$ and using \eqref{ba1}$_2$. The proof is complete.
\end{proof}

From Lemma 2.1, the following   estimates can be proved.

\begin{lemma}\label{2.2}
Let \eqref{kappa} and \eqref{assumption1} hold. Then
\begin{equation}\label{rho11}
\begin{split}
&  C^{-1}\leq \rho \leq C, \\
& \theta\geq C,\\
 &\iint_{Q_T} \frac{\kappa\theta_x^2}{\theta^{1+\alpha}} dxdt
\leq C,\quad\forall \alpha \in (0, \min\{1,q\}),\\
 &\int_0^T\|\theta\|_{L^\infty(\Omega)}^{q+1-\alpha}dt
\leq C,\quad\forall \alpha \in (0, \min\{1,q\}),\\
&\iint_{Q_T} \left(\lambda
u_x^2+\mu|\mathbf{w}_x|^2+\nu|\mathbf{b}_x|^2 \right)dxdt
\leq C,\\
&\int_0^T\|\mathbf{b}\|_{L^\infty(\Omega)}^{2}dt
\leq C,\\
&\iint_{Q_T} |\theta_x|^{3/2} dxdt\leq C.
\end{split}
\end{equation}
\end{lemma}
\begin{proof}
The proofs to the estimates $\rho \leq C$ and
\eqref{rho11}$_3$-\eqref{rho11}$_5$ can be found in \cite{FHL} where
the vacuum is permitted. \eqref{rho11}$_6$ is an immediate
consequence of   \eqref{rho11}$_5$, so  the estimate  $\rho\geq
C^{-1}$ can be proved in a standard way (see \cite{FanJiang}). We
omit their proofs for brevity.

Now we turn to \eqref{rho11}$_2$, whose proof depends  only  on the estimate $\rho \leq C$.  It follows from  \eqref{e1}$_5$
 that
\begin{equation*}
 \begin{split}
\theta_t+u\theta_x-\frac{1}{\rho} (\kappa \theta_x)_x
\geq & \frac{\lambda}{\rho} \left(u_x^2 -\frac{p}{\lambda} u_x \right)
=   \frac{\lambda}{\rho} \left(u_x
-\frac{p}{2\lambda}\right)^2-\frac{\gamma^2}{4\lambda}\rho\theta^2.
\end{split}
\end{equation*}
By $ \rho \leq C$, we have
 \begin{equation*}\begin{split}
\theta_t+u\theta_x-\frac{1}{\rho} (\kappa
\theta_x)_x   +K \theta^2\geq 0,\\
\end{split}\end{equation*}
where $K$ is a positive constant independent of $\mu$. Let $z=\theta-\underline\theta$, where
$\underline\theta=\frac{\min_{\overline\Omega}\theta_0}{Ct+1}$
 with $C=K\min_{\overline\Omega}\theta_0$. Then
 $ z_x|_{x=0, 1}=0, ~z|_{t=0}\geq0, $ and
 \begin{equation*}\begin{split}
&z_t+u z_x-\frac{1}{\rho} (\kappa z_x)_x  +K(\theta+\underline\theta)z\\
&=\theta_t+C\frac{\min_{\overline\Omega}\theta_0}{(Ct+1)^2}+u\theta_x-\frac{1}{\rho} (\kappa
\theta_x)_x  +K \theta^2- K\left(\frac{\min_{\overline\Omega}\theta_0}{Ct+1}\right)^2\\
&\geq C\frac{\min_{\overline\Omega}\theta_0}{(Ct+1)^2} -K\left(\frac{\min_{\overline\Omega}\theta_0}{Ct+1}\right)^2 = 0,
\end{split}\end{equation*}
and then,  $z\geq 0$ on
$\overline Q_T$ by the comparison theorem, so  \eqref{rho11}$_2$.

It remains to show \eqref{rho11}$_7$. By  \eqref{rho11}$_2$ and  \eqref{rho11}$_3$, we  have
\begin{equation}\label{theta01}
\begin{split}
\iint_{Q_T}\frac{\theta_x^2}{\theta} dxdt\leq C.
\end{split}
\end{equation}
Then, we have by the mean value theorem, Lemma \ref{2.1}, \eqref{rho11}$_1$ and H\"{o}lder's inequality
  \begin{equation*}
\begin{split}
\theta \leq & \int_\Omega \theta dx
+ \int_\Omega |\theta_x|dx\\
\leq &C+C\left(\int_\Omega\frac{\theta_x^2}{\theta}dx\right)^{1/2}\left(\int_\Omega
\theta  dx\right)^{1/2}\\
 \leq & C+C\left(\int_\Omega\frac{\theta_x^2}{\theta}dx\right)^{1/2},
\end{split}
\end{equation*}
which together with \eqref{theta01} gives
\begin{equation}\label{theta00}
\begin{split}
 \int_0^T\|\theta\|_{L^\infty(\Omega)}^2dt \leq C.\\
\end{split}
\end{equation}
Thus, it follows from H\"{o}lder's inequality, Lemma 2.1 and \eqref{theta00}  that
\begin{equation*}
\begin{split}
\iint_{Q_T}|\theta_x|^{3/2}dxdt\leq&\left(\iint_{Q_T}\frac{\theta_x^2}{\theta} dxdt\right)^{3/4}
\left(\iint_{Q_T}\theta^3 dxdt\right)^{1/4}\\
\leq& C\left(\int_0^T\|\theta^{2}\|_{L^\infty(\Omega)}\int_\Omega \theta dxdt\right)^{1/4}\leq C.\\
\end{split}
\end{equation*}
The proof is  complete.
\end{proof}

About the magnetic field $\mathbf{b}$, we have

\begin{lemma}\label{2.3}Let \eqref{kappa} and \eqref{assumption1} hold. Then
\begin{equation*}\label{b00}
\begin{split}
& \sup\limits_{0<t<T}\int_\Omega|\mathbf{b}|^4dx+\iint_{Q_T}|\mathbf{b}|^2|\mathbf{b}_x|^2dxdt\leq C.
\end{split}
\end{equation*}
\end{lemma}
\begin{proof}Multiplying \eqref{e1}$_4$ by $4|\mathbf{b}|^2\mathbf{b}$ and integrating over $Q_t$, we obtain
\begin{equation}\label{b0}
\begin{split}
&\int_\Omega|\mathbf{b}|^4dx+4\nu\iint_{Q_t}|\mathbf{b}|^2|\mathbf{b}_x|^2dxds+8\nu\iint_{Q_t}|\mathbf{b}\cdot \mathbf{b}_x|^2dxds\\
&=\int_\Omega|\mathbf{b}_0|^4dx+4\iint_{Q_t} \mathbf{w}_x \cdot(|\mathbf{b}|^2\mathbf{b})dxds-4\iint_{Q_t} (u \mathbf{b})_x\cdot(|\mathbf{b}|^2\mathbf{b})dxds.\\
\end{split}
\end{equation}

Integrating by parts and using Young's inequality, we have
\begin{equation}\label{b1}
\begin{split}
  \iint_{Q_t} \mathbf{w}_x \cdot(\mathbf{b} |\mathbf{b}|^2) dxds
 &=-\iint_{Q_t} \mathbf{w} \cdot(\mathbf{b}_x |\mathbf{b}|^2)dxds-2\iint_{Q_t} (\mathbf{w} \cdot\mathbf{b} )(\mathbf{b}\cdot \mathbf{b}_x)dxds\\
 &\leq 3\iint_{Q_t} |\mathbf{w}| |\mathbf{b}|^2 |\mathbf{b}_x| dxds\\
 &\leq \frac{\nu}{4}\iint_{Q_t} |\mathbf{b}|^2 |\mathbf{b}_x|^2dxds+C\iint_{Q_t} |\mathbf{w}|^2 |\mathbf{b}|^2 dxds\\
 &\leq \frac{\nu}{4}\iint_{Q_t} |\mathbf{b}|^2 |\mathbf{b}_x|^2dxds+C\int_0^t\|\mathbf{b}\|_{L^\infty(\Omega)}^2 \int_\Omega |\mathbf{w}|^2dxds\\
 &\leq\frac{\nu}{4}\iint_{Q_t} |\mathbf{b}|^2 |\mathbf{b}_x|^2dxds+C,
\end{split}
\end{equation}
where we used \eqref{ba1}$_2$ and \eqref{rho11}$_6$.

On the other hand, we have
\begin{equation}\label{b2}
\begin{split}
 &-\iint_{Q_t} (u \mathbf{b})_x\cdot|\mathbf{b}|^2\mathbf{b}dxds=3\iint_{Q_t}  u  (\mathbf{b}_x \cdot\mathbf{b})|\mathbf{b}|^2 dxds\\
 &\leq  \frac{\nu}{4}\iint_{Q_t} |\mathbf{b}|^2  |\mathbf{b}_x|^2 dxds+C\iint_{Q_t}   u^2 |\mathbf{b}|^4 dxds\\
 &\leq\frac{\nu}{4}\iint_{Q_t} |\mathbf{b}|^2  |\mathbf{b}_x|^2 dxds+C\int_0^t\|u^2\|_{L^\infty(\Omega)}\int_\Omega |\mathbf{b}|^4 dxds.
\end{split}
\end{equation}
Plugging \eqref{b1} and \eqref{b2} into \eqref{b0} and using
Gronwall's inequality, we finish the proof by noticing
$\int_0^T\|u^2\|_{L^\infty(\Omega)} dt\leq C\iint_{Q_T}u_x^2dxdt\leq
C$.
\end{proof}

\begin{lemma}\label{2.4}Let \eqref{kappa} and \eqref{assumption1} hold. Then\label{densitydevirative}
\begin{equation}\label{rho}
\begin{split}
&\sup\limits_{0<t<T}\int_\Omega\rho_x^2dx+\iint_{Q_T} \big(\rho_t^2+\theta
\rho_x^2\big) dxdt\leq C,\\[1mm]
&\left|\rho(x,t)-\rho(y,s)\right|\leq C\left(|x-y|^{1/2}
+|s-t|^{1/4}\right),\quad\forall (x, t), (y, s) \in \overline Q_T.
\end{split}
\end{equation}
\end{lemma}

\begin{proof}
Set $\eta=1/\rho$. It follows from the equation \eqref{e1}$_1$ that
$$
u_x=\rho(\eta_t+u\eta_x).
$$
Substituting it into \eqref{e1}$_2$ yields
\begin{equation*}
\left[\rho(u-\lambda\eta_x)\right]_t+\left[\rho
u(u-\lambda\eta_x)\right]_x
=\gamma\rho^2(\theta\eta_x-\eta\theta_x)-\mathbf{b}\cdot\mathbf{b}_x.
\end{equation*}
Multiplying it by $(u-\lambda\eta_x)$ and integrating over $Q_t$, we
have
\begin{equation*}
\begin{split}
&\frac12\int_\Omega\rho
(u-\lambda\eta_x)^2dx+\gamma\lambda\iint_{Q_t}
\theta\rho^2\eta_x^{2}dxds\\
&=\frac12\int_\Omega\rho_0(u_0+ \lambda\rho_0^{-2}\rho_{0x})^2dx+\gamma\iint_{Q_t}\rho^2\theta u  \eta_x dxds
\\
&\quad-\gamma\iint_{Q_t}\rho^2\eta\theta_x (u-\lambda\eta_x)dxds-\iint_{Q_t}\mathbf{b}\cdot\mathbf{b}_x(u-\lambda\eta_x)dxds.
\end{split}
\end{equation*}
To estimate the second integral on right-hand side, we use Young's
inequality,  Lemmas 2.1 and 2.2 to obtain
\begin{equation*}
\begin{split}
 \gamma\iint_{Q_t}\rho^2\theta u  \eta_x dxds
 &\leq
\frac{\gamma\lambda}{2}\iint_{Q_t}\theta\rho^{2}\eta_x^2dxds+C\iint_{Q_t} \theta u^2 dxds\\
&\leq  \frac{\gamma\lambda}{2}\iint_{Q_t}\theta\rho^{2}\eta_x^2dxds+C\int_0^t\|\theta\|_{L^\infty(\Omega)}
\int_{\Omega} u^2 dxds\\
&\leq
C+\frac{\gamma\lambda}{2}\iint_{Q_t}\theta\rho^{2}\eta_x^2dxds.
\end{split}
\end{equation*}
On the other hand, we have by Cauchy-Schwarz's inequality, \eqref{theta01} and Lemma \ref{2.3}
\begin{equation*}
\begin{split}
&-\gamma\iint_{Q_t}\rho^2\eta\theta_x (u-\lambda\eta_x)dxds-\iint_{Q_t}\mathbf{b}\cdot\mathbf{b}_x(u-\lambda\eta_x)dxds\\
&\leq  C+C\iint_{Q_t}\theta \rho
(u-\lambda\eta_x)^2dxds+C\iint_{Q_t}\frac{\theta_x^2}{\theta}dxds+C\iint_{Q_t} \rho(u-\lambda\eta_x)^2dxds\\
&\leq  C+C\int_0^t\big(1+\|\theta\|_{L^\infty(\Omega)}\big)\int_\Omega\rho
(u-\lambda\eta_x)^2dxds.
\end{split}
\end{equation*}
Combining the above results yields
\begin{equation*}
\begin{split}
  & \int_\Omega\rho
(u-\lambda\eta_x)^2dx+ \iint_{Q_t} \theta\rho^2\eta_x^{2}dxds\\
&\leq C
+C\int_0^t\big(1+\|\theta\|_{L^\infty(\Omega)}\big)\int_\Omega\rho
(u-\lambda\eta_x)^2dxds,
\end{split}
\end{equation*}
 which together with  Gronwall's inequality gives
 \begin{equation*}\label{rho22}
\begin{split}
&\sup\limits_{0<t<T}\int_\Omega\rho_x^2dx+\iint_{Q_T}  \theta \rho_x^2 dxds\leq C .
\end{split}
\end{equation*}
 By this estimate  and Lemma 2.2, we derive from the equation \eqref{e1}$_1$ that
\begin{equation*}
\begin{split}
 \iint_{Q_T} \rho_t^2dxdt \leq & C\int_0^T\|u^2\|_{L^\infty(\Omega)}\int_\Omega\rho_x^2dxdt+C\iint_{Q_T}u_x^2dxdt \leq C.\\
\end{split}
\end{equation*}
 Thus \eqref{rho}$_1$ holds.

  Now we prove the second estimate. Let $\beta(x)=\rho(x,t)-\rho(x,s)$ for any $x \in [0, 1]$
  and $s, t \in [0, T]$ with $s\neq t$. Then for any $x\in [0, 1]$ and $\delta \in (0, 1/2]$, there exist  some
  $y\in [0, 1]$ and $\xi$ between $x$ and $y$ such that $\delta=|y-x|$  and
  $\beta(\xi)=\frac{1}{x-y}\int^x_y\beta(z)dz$, and hence
\begin{equation*}
\begin{split}
\beta(x)=\frac{1}{x-y}\int^x_y\beta(z)dz+\int_\xi^x\beta'(z)dz,
\end{split}
\end{equation*}
therefore, by H\"{o}lder's inequality and \eqref{rho22},
\begin{equation*}
\begin{split}
|\beta(x)|\leq& \frac{1}{\delta}\left|\int^x_y\beta(z)dz\right|+\left|\int_\xi^x\beta'(z)dz\right|\\[1mm]
\leq &\frac{1}{\delta}\left|\int^x_y\hspace{-2mm}\int_s^t\rho_\tau d\tau dz\right|+\left|\int_\xi^x\left[\rho_z(z,t)-\rho_z(z,s)\right]dz\right|\\[1mm]
\leq &\frac{1}{\delta}\left(\iint_{Q_T}\rho_\tau^2 d\tau dz\right)^{1/2}|x-y|^{1/2}|s-t|^{1/2}\\[1mm]
&+ \left(\int_0^1(|\rho_z(z,s)|^2+|\rho_z(z,t)|^2)dz\right)^{1/2}|x-\xi|^{1/2}\\[1mm]
\leq & C\delta^{-1/2}|s-t|^{1/2}+C\delta^{1/2}.
\end{split}
\end{equation*}
If $0<|s-t|^{1/2}<1/2$, taking $\delta=|s-t|^{1/2}$ yields
\begin{equation}\label{rho6}
\begin{split}
|\rho(x,s)-\rho(x,t)|\leq C|s-t|^{1/4}.
\end{split}
\end{equation}
If $ |s-t|^{1/2}\geq 1/2$, then \eqref{rho6} holds since $\rho$ is uniformly bounded in $\mu$.

On the other hand, we have  by \eqref{rho}$_1$
\begin{equation}
\begin{split}
|\rho(x,t)-\rho(y,t)|=\left|\int_y^x\rho_zdz\right|\leq C|x-y|^{1/2}.
\end{split}
\end{equation}
Thus, \eqref{rho}$_2$ is a consequence of the triangle inequality. The proof is complete.
\end{proof}

To deduce other required $\mu$-uniform estimates, we need the
following  lemma which plays an important role in this paper.

\begin{lemma}\label{2.5}Let \eqref{kappa} and \eqref{assumption1} hold. Then
\begin{equation}\label{uxx}
\begin{split}
\iint_{Q_T}|u_{xx}|^{m_0} dxdt\leq C,\quad m_0=\min\{m, 4/3\}.
\end{split}
\end{equation}
In particular,
\begin{equation}\label{u0}
\begin{split}
\int_0^T\|u_{x}\|_{L^\infty(\Omega)}^{m_0}dt \leq C.
\end{split}
\end{equation}
\end{lemma}
\begin{proof}Note that the estimate \eqref{u0} is an immediate consequence of \eqref{uxx}. Thus, it is enough to prove \eqref{uxx}. To this end, we rewrite the equation \eqref{e1}$_2$ as
\begin{equation}\label{f2}
\begin{split}
&
 u_t-\frac{\lambda}{\rho}u_{xx}=-uu_x-
 \gamma\theta_x-\frac{\gamma}{\rho}\rho_x\theta-\frac{1}{\rho}\mathbf{b}\cdot\mathbf{b}_x=:f.
\end{split}
\end{equation}
We will apply $L^p$ estimates of linear parabolic equations  (cf.\cite[Theorem 7.17]{Lie}) to show \eqref{uxx}. From \eqref{rho}$_2$,
the coefficient $a(x,t):=\lambda/\rho$ satisfies
$$\left|a(x,t)-a(y,s)\right|\leq C\left(|x-y|^{1/2}
+|s-t|^{1/4}\right),\quad\forall (x, t), (y, s) \in \overline Q_T.
$$
Due to the condition on $u_0$ in \eqref{assumption1}, we only need
to give a uniform bound of $f$ in $L^{4/3}(Q_T)$.

From Lemmas 2.2 and 2.3,  the second term and the forth term on right-hand side of \eqref{f2} are
 uniformly bounded in $L^{3/2}(Q_T)$ and $L^{2}(Q_T)$, respectively.

To deal with the first term on right-hand side of \eqref{f2}, we observe by H\"{o}lder inequality and Lemma 2.1
\begin{equation*}
\begin{split}
 u^2&\leq2\int_\Omega|uu_x|dx \leq 2\left(\int_\Omega u^2dx\right)^{1/2}\left(\int_\Omega u_x^2dx\right)^{1/2} \leq C\left(\int_\Omega u_x^2dx\right)^{1/2},\\
\end{split}
\end{equation*}
therefore, we have by Lemma 2.2
\begin{equation*}
\begin{split}
\int_0^T\|u\|_{L^\infty}^4 dt
 \leq C \iint_{Q_T} u_x^2dx\leq C,
\end{split}
\end{equation*}
which together with Young's inequality yields
\begin{equation*}
\begin{split}
  \iint_{Q_T}|uu_x|^{3/2}dxdt &\leq  C\iint_{Q_T}u_x^2dxdt+C\iint_{Q_T}u^6dxdt\\
 &\leq C+ C \int_0^T\|u\|_{L^\infty(\Omega)}^4 \int_{\Omega}u^2dxdt \leq C.
\end{split}
\end{equation*}

As to the third term on right-hand side of \eqref{f2}, we have by \eqref{rho}$_1$ and \eqref{theta00}
\begin{equation*}
\begin{split}
  \iint_{Q_T}|\rho_x \theta|^{4/3}dxdt
 &\leq C\iint_{Q_T}\rho_x^2\theta dxdt+C\iint_{Q_T}\theta^{2}dxdt
 \leq C.
\end{split}
\end{equation*}

Combining the above results gives $\|f\|_{L^{4/3}(Q_T)}\leq C.$ The
proof is then completed.
\end{proof}

By a direct application of the above lemma, we obtain
\begin{lemma}\label{2.6}
Let \eqref{kappa} and \eqref{assumption1} hold. Then
\begin{equation*}\label{we11}
\begin{split}
& \mu\sup\limits_{0<t<T}\int_\Omega|\mathbf{w}_x|^2dx+\mu^2 \iint_{Q_T}
 |\mathbf{w}_{xx}|^2 dxds  \leq C.
\end{split}
\end{equation*}
 \end{lemma}
\begin{proof}
We rewrite  \eqref{e1}$_3$ in the form
\begin{equation}\label{w12}
\begin{split}
 \mathbf{w}_t-\frac{\mu}{\rho}\mathbf{w}_{xx}=\frac{1}{\rho}\mathbf{b}_x-u\mathbf{w}_x,
\end{split}
\end{equation}
and multiply it by $\mu\mathbf{w}_{xx}$ and integrating over $Q_t$
to obtain
\begin{equation}\label{0v100}
\begin{split}
 & \frac\mu2\int_\Omega|\mathbf{w}_x|^2dx
 +\mu^2\iint_{Q_t}\frac{1}{\rho}|\mathbf{w}_{xx}|^2dxdx\\
  &=\frac\mu2\int_\Omega|\mathbf{w}_{0x}|^2dx-\mu\iint_{Q_t}\frac{1}{\rho}\mathbf{b}_x\cdot\mathbf{w}_{xx}dxds\\
&\quad -\frac{\mu}{2}\iint_{Q_t}u_x |\mathbf{w}_x|^2dxds+\mu\int_0^t \mathbf{w}_{t}\cdot\mathbf{w}_{x} \Big|_{x=0}^{x=1} ds\\
& \leq C\mu+\frac{\mu^2}{4}\iint_{Q_t}\frac{1}{\rho}|\mathbf{w}_{xx}|^2dxds+C\iint_{Q_t}|\mathbf{b}_x|^2dxds\\
&\quad+ C\int_0^t\|u_x\|_{L^\infty(\Omega)}\left(\mu\int_\Omega|\mathbf{w}_x|^2dx\right)ds+C\mu\int_0^t\|\mathbf{w}_x\|_{L^\infty(\Omega)} ds.
  \end{split}
\end{equation}
From the mean value theorem and H$\ddot{o}$lder inequality, we obtain
\begin{equation}\label{wx2}
\begin{split}
   |\mathbf{w}_x|^2\leq & \Big|\frac{\mathbf{w}(b,t)-\mathbf{w}(a,t)}{b-a}\Big|^2+2\int_\Omega|\mathbf{w}_{x}||\mathbf{w}_{xx}|dx\\
   \leq& C+C\left(\int_\Omega|\mathbf{w}_{x}|^2dx\right)^{1/2}\left(\int_\Omega|\mathbf{w}_{xx}|^2dx\right)^{1/2},
\end{split}
\end{equation}
and so,  Young's inequality yields
\begin{equation*}\label{w5}
\begin{split}
   \mu\int_0^t\|\mathbf{w}_x\|_{L^\infty(\Omega)} ds
   \leq& C\mu+C\int_0^t\mu^{1/4}\left(\mu\int_\Omega|\mathbf{w}_{x}|^2dx\right)^{1/4}
   \left(\mu^2\int_\Omega|\mathbf{w}_{xx}|^2dx\right)^{1/4}ds\\[1mm]
\leq &C\sqrt{\mu}+\frac{C  \mu}{\epsilon}\iint_{Q_t}
|\mathbf{w}_{x}|^2dxds+\epsilon \mu^2
\iint_{Q_t}\frac{1}{\rho}|\mathbf{w}_{xx}|^2dxds, \forall \epsilon \in (0, 1).
\end{split}
\end{equation*}
Inserting  it into \eqref{0v100} and taking a small $\epsilon>0$, we find that
\begin{equation*}
\begin{split}
  & \mu\int_\Omega|\mathbf{w}_x|^2dx
 +\mu^2\iint_{Q_t}\frac{1}{\rho}|\mathbf{w}_{xx}|^2dxds \\
 &\leq C  +C\int_0^t\big(1+\|u_x\|_{L^\infty}\big)\left(\mu\int_{Q_t}|\mathbf{w}_x|^2dx\right)ds.
  \end{split}
\end{equation*}
Thus, the lemma follows from Gronwall's inequality and  \eqref{u0}. This proof is complete.
\end{proof}

Our next main task is to show the other estimates appearing in
Theorem 1.1. To this end, we need  three preliminary lemmas. The
first one reads as

\begin{lemma}\label{2.7} Let \eqref{kappa} and \eqref{assumption1} hold. Then
\begin{equation*}
\begin{split}
    & \int_\Omega|\mathbf{b}_x|^2\omega^2dx+ \iint_{Q_t}|\mathbf{b}_{xx}|^2\omega^2dxds
    \leq C\iint_{Q_t}  |\mathbf{w}_{x}|^2\omega^2  dxds+C\left(\iint_{Q_t} u_{xx}^2
    dxds\right)^{1/2},
  \end{split}
\end{equation*}
where $\omega$ is the same as that in Theorem 1.1.
 \end{lemma}
 \begin{proof}
Multiplying \eqref{e1}$_4$ by $\mathbf{b}_{xx}\omega^2(x)$ and
integrating over $Q_t$, we have
\begin{equation}\label{b111}
\begin{split}
   &-\iint_{Q_t} \mathbf{b}_t\cdot \mathbf{b}_{xx}\omega^2 dxdt+\nu\iint_{Q_t} |\mathbf{b}_{xx}|^2\omega^2 dxds\\
    &=\iint_{Q_t}  (u\mathbf{b})_x\cdot \mathbf{b}_{xx}\omega^2  dxds-\iint_{Q_t} \mathbf{w}_x\cdot \mathbf{b}_{xx}\omega^2 dxds.
  \end{split}
\end{equation}
To estimate the first integral on left-hand side of \eqref{b111}, we
integrate by parts and use \eqref{e1}$_4$ to obtain
\begin{equation}\label{be1}
\begin{split}
   \iint_{Q_t} \mathbf{b}_t\cdot \mathbf{b}_{xx}\omega^2 dxdt
   &=-\frac{1}{2}\int_\Omega|\mathbf{b}_x|^2\omega^2   dx
   +\frac{1}{2}\int_\Omega|\mathbf{b}_{0x}|^2\omega^2   dx-2\iint_{Q_t} \mathbf{b}_t\cdot \mathbf{b}_{x}\omega\omega' dxdt\\
   &=-\frac{1}{2}\int_\Omega|\mathbf{b}_x|^2\omega^2dx
   +\frac{1}{2}\int_\Omega|\mathbf{b}_{0x}|^2\omega^2dx\\
    &\quad-2\iint_{Q_t} \left(\nu\mathbf{b}_{xx}+\mathbf{w}_x-u\mathbf{b}_x-u_x\mathbf{b}\right)\cdot \mathbf{b}_{x}\omega\omega'  dxds.
  \end{split}
\end{equation}
Below we deal with the third term on right-hand side of
\eqref{be1}. By Cauchy-Schwarz's inequality and Lemmas \ref{2.2} and \ref{2.3}, we have
\begin{equation*}
\begin{split}
    &-2\iint_{Q_t} \left(\nu\mathbf{b}_{xx}+\mathbf{w}_x-u_x\mathbf{b}\right)\cdot \mathbf{b}_{x}\omega\omega'  dxds\\
   &\leq  \frac{\nu}{4}\iint_{Q_t}|\mathbf{b}_{xx}|^2\omega^2dxds  +C\iint_{Q_t}|\mathbf{b}_{x}|^2dxds+C\iint_{Q_t}  |\mathbf{w}_{x}|^2\omega^2  dxds\\
    & \quad+  C\iint_{Q_t} u_x^2dxds+C\iint_{Q_t} |\mathbf{b}\cdot\mathbf{b}_x|^2 dxds\\
    &\leq C+ \frac{\nu}{4}\iint_{Q_t}|\mathbf{b}_{xx}|^2\omega^2dxds  +C\iint_{Q_t}  |\mathbf{w}_{x}|^2\omega^2  dxds.\\
  \end{split}
\end{equation*}
Observe that since from the mean value theorem and $u(1,t)=u(0,t)=0$, we have
\begin{equation}\label{u3}
\begin{split}
|u(x,t)|\leq \|u_x\|_{L^\infty(\Omega)}\omega(x),
\end{split}
\end{equation}
so
\begin{equation*}
\begin{split}
    &2\iint_{Q_t} u |\mathbf{b}_{x}|^2\omega\omega'dxds
    \leq  C\int_{0}^t \|u_x\|_{L^\infty(\Omega)} \int_\Omega|\mathbf{b}_{x}|^2\omega^2dxds.\\
  \end{split}
\end{equation*}
Substituting them into \eqref{be1} yields
\begin{equation}\label{be3}
\begin{split}
     \iint_{Q_t} \mathbf{b}_t\cdot \mathbf{b}_{xx}\omega^2 dxdt  \leq & C-\frac{1}{2}\int_\Omega|\mathbf{b}_x|^2\omega^2dx+\frac{\nu}{4}\iint_{Q_t}|\mathbf{b}_{xx}|^2\omega^2dxds\\
& +C\int_{0}^t\|u_x\|_{L^\infty(\Omega)}\int_\Omega|\mathbf{b}_{x}|^2\omega^2dxds
 +C\iint_{Q_t}  |\mathbf{w}_{x}|^2\omega^2  dxds.
  \end{split}
\end{equation}
As to the two terms on right-hand side of \eqref{b111}, we have by  Young's inequality
\begin{equation}\label{be2}
\begin{split}
   & \iint_{Q_t}  (u\mathbf{b})_x\cdot \mathbf{b}_{xx}\omega^2  dxds-\iint_{Q_t} \mathbf{w}_x\cdot \mathbf{b}_{xx}\omega^2 dxds\\
   &\leq  \frac{\nu}{4}\iint_{Q_t} |\mathbf{b}_{xx}|^2\omega^2  dxds+C\iint_{Q_t}  |(u\mathbf{b})_x|^2\omega^2  dxds
   +C\iint_{Q_t}  |\mathbf{w}_{x}|^2\omega^2  dxds.
  \end{split}
\end{equation}
It remains to treat the second term on right-hand side of \eqref{be2}. We observe   by Lemma \ref{2.2}
\begin{equation}\label{ux}
\begin{split}
     &\int_0^t\|u_x\|_{L^\infty(\Omega)}^2 ds\leq C\iint_{Q_t}|u_xu_{xx}|dxds  \leq C\left(\iint_{Q_t}u_{xx}^2 dxds\right)^{1/2},
  \end{split}
\end{equation}
which together with Lemma 2.1 gives
\begin{equation*}
\begin{split}
    \iint_{Q_t}  |(u\mathbf{b})_x|^2\omega^2  dxds \leq &C\iint_{Q_t} u^2|\mathbf{b}_{x}|^2\omega^2  dxds+C\iint_{Q_t} u_x^2|\mathbf{b} |^2\omega^2  dxds\\
  \leq & C\int_0^t\|u^2\|_{L^\infty(\Omega)}\int_\Omega|\mathbf{b}_{x}|^2\omega^2  dxds+C\left(\iint_{Q_t} u_{xx}^2 dxds\right)^{1/2}.
  \end{split}
\end{equation*}
Substituting it into \eqref{be2} and then, substituting the resulting inequality and \eqref{be3} into \eqref{b111} and using Gronwall's inequality, we finish the proof.
\end{proof}

\begin{lemma}\label{2.8}Let \eqref{kappa} and \eqref{assumption1} hold. Then
\begin{equation}\label{w1}
\begin{split}
   &  \int_\Omega|\mathbf{w}_x|^2\omega^2    dx+\mu \iint_{Q_t} |\mathbf{w}_{xx}|^2\omega^2 dxds\leq C+C\left(\iint_{Q_t} u_{xx}^2 dxds\right)^{1/2},\\
   &\iint_{Q_t}\big(|\mathbf{w}_t|^2+u^2|\mathbf{w}_x|^2\big)dxdt\leq C+C \iint_{Q_t}u_{xx}^2 dxds.
  \end{split}
\end{equation}
\end{lemma}
\begin{proof}
  Multiplying \eqref{w12} by $\mathbf{w}_{xx}\omega^2(x)$ and integrating over $Q_t$, we have
\begin{equation}\label{0v10}
\begin{split}
   &-\iint_{Q_t} \mathbf{w}_t\cdot \mathbf{w}_{xx}\omega^2 dxdt+\mu\iint_{Q_t}|\mathbf{w}_{xx}|^2\frac{\omega^2}{\rho}dxds\\
    &=\iint_{Q_t} u\mathbf{w}_x\cdot \mathbf{w}_{xx}\omega^2 dxds-\iint_{Q_t}  \mathbf{b}_x\cdot \mathbf{w}_{xx}\frac{\omega^2}{\rho}  dxds.
  \end{split}
\end{equation}
Integrating by parts and using \eqref{w12}, we have
\begin{equation*}
\begin{split}
   & \iint_{Q_t} \mathbf{w}_t\cdot \mathbf{w}_{xx}\omega^2 dxdt\\[1mm]
   &=-\frac{1}{2}\int_\Omega|\mathbf{w}_x|^2\omega^2   dx
  +\frac{1}{2}\int_\Omega|\mathbf{w}_{0x}|^2\omega^2   dx-2\iint_{Q_t} \mathbf{w}_t\cdot \mathbf{w}_{x}\omega\omega' dxdt\\[1mm]
   &=-\frac{1}{2}\int_\Omega|\mathbf{w}_x|^2\omega^2  dx
   +\frac{1}{2}\int_\Omega|\mathbf{w}_{0x}|^2\omega^2   dx\\[1mm]
   &\quad-2\iint_{Q_t} \left(\frac{\mu}{\rho}\mathbf{w}_{xx}-u\mathbf{w}_x
   +\frac{\mathbf{b}_x}{\rho}\right)\cdot \mathbf{w}_{x}\omega\omega'  dxds\\[1mm]
   &\leq C-\frac{1}{2}\int_\Omega|\mathbf{w}_x|^2\omega^2   dx+C\mu^2\iint_{Q_t}|\mathbf{w}_{xx}|^2dxds +C\iint_{Q_t}|\mathbf{w}_x|^2\omega^2   dxds\\[1mm]
   &\quad+C\iint_{Q_t}|\mathbf{b}_{x}|^2dxds+C\iint_{Q_t}|u||\mathbf{w}_x|^2\omega dxds.
  \end{split}
\end{equation*}
From \eqref{u3} it follows that
\begin{equation*}
\begin{split}
    \iint_{Q_t}|u||\mathbf{w}_x|^2\omega dxds\leq \int_0^t\|u_x\|_{L^\infty(\Omega)} \int_\Omega|\mathbf{w}_x|^2\omega^2dxds,\\
  \end{split}
\end{equation*}
which together with Lemmas \ref{2.2} and \ref{2.6} gives
\begin{equation*}
\begin{split}
   & \iint_{Q_t} \mathbf{w}_t\cdot \mathbf{w}_{xx}\omega^2 dxdt\\
&   \leq C-\frac{1}{2}\int_\Omega|\mathbf{w}_x|^2\omega^2   dx
   +C\int_0^t\big(1+\|u_x\|_{L^\infty(\Omega)} \big)\int_\Omega|\mathbf{w}_x|^2\omega^2dxds.\\
  \end{split}
\end{equation*}
To estimate the right-hand side of \eqref{0v10}, we have by integrating
by parts  and by \eqref{u3}
\begin{equation*}
\begin{split}
  \iint_{Q_t} u\mathbf{w}_x\cdot \mathbf{w}_{xx}\omega^2 dxds
  =&-\frac{1}{2}\iint_{Q_t} |\mathbf{w}_x|^2[u_x\omega^2+2u\omega\omega']
  dxds\\
  \leq&C\int_0^t \|u_x\|_{L^\infty(\Omega)} \int_\Omega|\mathbf{w}_x|^2\omega^2dxds,\\
  \end{split}
\end{equation*}
and
\begin{equation*}
\begin{split}
   &-\iint_{Q_t} \mathbf{b}_x\cdot \mathbf{w}_{xx}\frac{\omega^2}{\rho} dxds \\[1mm]
   &=\iint_{Q_t}\mathbf{w}_x\cdot \mathbf{b}_{xx}  \frac{\omega^2}{\rho}dxds+2\iint_{Q_t} \mathbf{w}_x\cdot \mathbf{b}_{x} \frac{\omega\omega'}{\rho} dxds-\iint_{Q_t} \mathbf{w}_x\cdot \mathbf{b}_{x}\frac{\omega^2\rho_x}{\rho^2} dxds\\[1mm]
& \leq C\iint_{Q_t}  |\mathbf{b}_{xx}|^2\omega^2 dxds+C\iint_{Q_t}  |\mathbf{w}_x|^2\omega^2 dxds+C\iint_{Q_t}  |\mathbf{b}_{x}|^2dxds\\
&\quad+C\iint_{Q_t}  |\mathbf{b}_x|^2\omega^2 \rho_x^2dxds\\[1mm]
&\leq  C+C\iint_{Q_t}  |\mathbf{w}_x|^2\omega^2 dxds+C\iint_{Q_t}  |\mathbf{b}_{xx}|^2\omega^2 dxds,
  \end{split}
\end{equation*}
where we  used the fact by Lemmas \ref{2.2} and \ref{2.4}
\begin{equation*}
\begin{split}
     \iint_{Q_t}  |\mathbf{b}_x|^2\omega^2 \rho_x^2dxds\leq &C\int_0^t  \left\||\mathbf{b}_x|^2\omega^2\right\|_{L^\infty(\Omega)}  ds\\[1mm]
      \leq & C\iint_{Q_t}  \left|(|\mathbf{b}_x|^2\omega^2)_x\right| dxds\\[1mm]
    \leq & C\iint_{Q_t}    |\mathbf{b}_{x}|^2|\omega\omega'|dxds+C \iint_{Q_t}    |\mathbf{b}_x\cdot\mathbf{b}_{xx}|  \omega^2 dxds\\[1mm]
    \leq &C+C\iint_{Q_t}    |\mathbf{b}_{xx}|^2 \omega^2 dxds.
  \end{split}
\end{equation*}
Substituting the above results into \eqref{0v10} and using Lemma \ref{2.7}, we have
\begin{equation*}
\begin{split}
     &\int_\Omega|\mathbf{w}_x|^2\omega^2   dx+\mu \iint_{Q_t} |\mathbf{w}_{xx}|^2\omega^2   dxds\\
     & \leq C + C \int_0^t\big(1+\|u_x\|_{L^\infty(\Omega)}\big)\int_\Omega|\mathbf{w}_x|^2\omega^2  dxds+C\left(\iint_{Q_t} u_{xx}^2 dxds\right)^{1/2}.
  \end{split}
\end{equation*}
Thus, the first estimate of this lemma follows from Gronwall's inequality and
\eqref{u0}.

Consequently, we have by  \eqref{u3}, the first estimate of this lemma and \eqref{ux}
\begin{equation*}\label{v11}
\begin{split}
    \iint_{Q_T} u^2|\mathbf{w}_x|^2 dxdt&\leq \int_0^T\|u_x\|_{L^\infty(\Omega)}^2\int_\Omega |\mathbf{w}_x|^2\omega^2 dxdt \\ &\leq C\int_0^T\|u_x\|_{L^\infty(\Omega)}^2 dt\left[1+\left(\iint_{Q_T}u_{xx}^2 dxdt\right)^{1/2}\right]\\
    &\leq C+C \iint_{Q_T}u_{xx}^2 dxdt.
  \end{split}
\end{equation*}
Furthermore, by Lemmas \ref{2.2} and \ref{2.6}, we derive from
\eqref{w12} that
\begin{equation*}
\begin{split}
   &\iint_{Q_T} |\mathbf{w}_t|^2dxdt\leq C +C\iint_{Q_T}u^2|\mathbf{w}_x|^2dxdt\leq C+C \iint_{Q_T}u_{xx}^2 dxdt.
  \end{split}
\end{equation*}
The proof is complete.
\end{proof}

\begin{lemma}\label{2.9} Let \eqref{kappa} and \eqref{assumption1} hold. Then
\begin{equation}\label{b99}
\begin{split}
  &  \int_\Omega|\mathbf{b}_{x}|^2dx+\iint_{Q_t} |\mathbf{b}_t|^2
  dxdt
  \leq C+ C\left(\iint_{Q_t}u_{xx}^2 dxds\right)^{1/2}.  \\
  \end{split}
\end{equation}
\end{lemma}
\begin{proof}
Multiplying \eqref{e1}$_4$ by $\mathbf{b}_t$ and
integrating over $Q_t$ yield
\begin{equation}\label{b9}
\begin{split}
  &\frac{\nu}{2}\int_\Omega|\mathbf{b}_{x}|^2dx+ \iint_{Q_t} |\mathbf{b}_t|^2
  dxdt
  =\frac{\nu}{2}\int_\Omega|\mathbf{b}_{0x}|^2dx+\iint_{Q_t}  \big[\mathbf{w}_x
  - (u\mathbf{b})_{x}\big]\cdot\mathbf{b}_t dxdt.  \\
  \end{split}
\end{equation}
Integrating by parts yields
\begin{equation*}
\begin{split}
  &\iint_{Q_t}   \mathbf{w}_x  \cdot\mathbf{b}_t dxdt= \int_\Omega\mathbf{w}_x  \cdot\mathbf{b}dx
  -\int_\Omega\mathbf{w}_{0x}  \cdot\mathbf{b}_0dx-  \iint_{Q_t}   (\mathbf{w}_t)_x  \cdot\mathbf{b}  dxdt\\
  &=-\int_\Omega\mathbf{w}_{0x}\cdot\mathbf{b}_0dx -\int_\Omega\mathbf{w}\cdot\mathbf{b}_xdx
  +\iint_{Q_t}   \mathbf{w}_t\cdot\mathbf{b}_x  dxdt,
  \end{split}
\end{equation*}
which together with  Lemmas \ref{2.1} and \ref{2.2} and
\eqref{w1}$_2$ gives
\begin{equation}\label{b10}
\begin{split}
   \iint_{Q_t}   \mathbf{w}_x  \cdot\mathbf{b}_t dxdt
  &  \leq C+\frac{\nu}{4}\int_\Omega|\mathbf{b}_x|^2dx+\left(\iint_{Q_t}|\mathbf{b}_x|^2dxdt\right)^{1/2}
\left(\iint_{Q_t}|\mathbf{w}_t|^2dxdt\right)^{1/2}\\
&\leq C +\frac{\nu}{4}\int_\Omega|\mathbf{b}_x|^2dx+C\left(\iint_{Q_t}|\mathbf{w}_t|^2dxdt\right)^{1/2}\\
&\leq C +\frac{\nu}{4}\int_\Omega|\mathbf{b}_x|^2dx+C\left(\iint_{Q_t}u_{xx}^2dxds\right)^{1/2}.
  \end{split}
\end{equation}
On the other hand, we have by Cauchy-Schwarz's inequality, Lemma 2.1 and \eqref{ux}
\begin{equation}\label{b11}
\begin{split}
   &-\iint_{Q_t}  (u\mathbf{b})_{x} \cdot\mathbf{b}_t dxdt\\
   & \leq \frac12  \iint_{Q_t} |\mathbf{b}_t|^2 dxdt
  +\frac12\iint_{Q_t}  |(u\mathbf{b})_{x}|^2 dxds\\
   &\leq   \frac12  \iint_{Q_t} |\mathbf{b}_t|^2 dxdt+C\int_0^t\|u^2\|_{L^\infty(\Omega)} \int_{\Omega}  |\mathbf{b}_x|^2 dxds+C\int_0^t\|u_x\|_{L^\infty(\Omega)}^2\int_{\Omega}  |\mathbf{b}|^2 dxds\\
  &\leq   \frac12  \iint_{Q_t} |\mathbf{b}_t|^2 dxdt+C\int_0^t\|u^2\|_{L^\infty(\Omega)} \int_{\Omega}  |\mathbf{b}_x|^2 dxds+C\left(\iint_{Q_t}u_{xx}^2 dxds\right)^{1/2}.
\end{split}\end{equation}
Substituting them into \eqref{b9} and using Gronwall's inequality, we complete the proof.
\end{proof}

Now we can prove the following desired results.

\begin{lemma}\label{2.10}
Let \eqref{kappa} and \eqref{assumption1} hold. Then $\|(u,\mathbf{b})\|_{L^\infty(Q_T)}\leq C$, and
\begin{equation}\label{w1112}
\begin{split}
&\sup\limits_{0<t<T}\int_\Omega(\rho_t^2+u_x^2+\theta^2)dx +\iint_{Q_T} \big( u_t^2 +u_{xx}^2+\kappa\theta_x^2\big)dxdt\leq C,\\
& \sup\limits_{0<t<T}\int_\Omega|\mathbf{w}_x|^2\omega^2 dx+\iint_{Q_T}\big(|u|^2|\mathbf{w}_x|^2+ |\mathbf{w}_{t}|^2 \big)  dxdt \leq C,\\
&\sup\limits_{0<t<T}\int_\Omega|\mathbf{b}_{x}|^2dx+\iint_{Q_T}\big(|\mathbf{b}_t|^2 + |\mathbf{b}_{xx}|^2\omega^2\big) dxdt\leq C.
\end{split}
\end{equation}
 \end{lemma}
 \begin{proof}
Rewrite the equation \eqref{e1}$_2$ in the form
 \begin{equation*}\label{0u0}
\begin{split}
&\sqrt{\rho}u_t-\frac{\lambda}{\sqrt{\rho}}u_{xx}=-\sqrt{\rho}uu_x-
\gamma\sqrt{\rho}\theta_x-\frac{\gamma}{\sqrt{\rho}}\rho_x\theta-\frac{1}{\sqrt{\rho}}\mathbf{b}\cdot\mathbf{b}_x.
\end{split}
\end{equation*}
Using Cauchy-Schwarz's inequality, we obtain
\begin{equation}\label{0u1}
\begin{split}
&  \frac{\lambda}2\int_\Omega u_x^2dx+\iint_{Q_t}\big(\rho u_t^2+\lambda^2\rho^{-1}u_{xx}^2\big)dxdt\\
&\leq \frac{\lambda}2\int_\Omega u_{0x}^2dx+C\iint_{Q_t} \big(u^2u_x^2+  \theta_x^2+\rho_x^2 \theta^2 + |\mathbf{b}\cdot\mathbf{b}_x|^2\big)dxds\\
&\leq C+C\int_0^t\|u^2\|_{L^\infty(\Omega)}\int_\Omega u_x^2dxds+C\iint_{Q_t} \theta_x^2dxds+C\int_0^t\|\theta^2\|_{L^\infty(\Omega)}\int_{\Omega} \rho_x^2  dxds\\
&\leq C+C\int_0^t\|u^2\|_{L^\infty(\Omega)}\int_\Omega u_x^2dxds+C\iint_{Q_t} \big(\theta_x^2+ \theta^2 \big) dxds,
\end{split}
\end{equation}
where we used  \eqref{rho} and $\int_0^t\|\theta^2\|_{L^\infty(\Omega)}ds\leq C \iint_{Q_t} (\theta^2+\theta_x^2)dxds$.

Our next step is to multiply \eqref{e1}$_5$   by $ \theta $ and
integrate over $Q_t$. We have
\begin{equation}\label{theta0}
\begin{split}
   &\frac12\int_\Omega \rho\theta^2dx+ \iint_{Q_t}  \kappa \theta_x^2dxds = -\iint_{Q_t}pu_x\theta dxds+\iint_{Q_t}  \theta  {\cal Q}   dxds.
  \end{split}
\end{equation}
By Cauchy-Schwarz's inequality and \eqref{theta00}, we obtain
\begin{equation*}
\begin{split}
     -\gamma\iint_{Q_t}  \rho\theta^2 u_xdxds
 &\leq C \iint_{Q_t} \theta^2dxds+C\iint_{Q_t}\theta^{2} u_x^2 dxds\\
&\leq C+C\int_0^t\|\theta\|_{L^\infty(\Omega)}^{2} \int_{\Omega}u_x^2
dxds.
  \end{split}
\end{equation*}
On the other hand, we have by Lemmas \ref{2.2}, \ref{2.6} and \ref{2.9}
\begin{equation*}
\begin{split}
     \iint_{Q_t}  \theta  {\cal Q} dxds
    &\leq  C\int_0^t\|\theta\|_{L^\infty(\Omega)} \left
    \{\int_{\Omega}\big(u_x^2 +\mu|\mathbf{w}_x|^2dx+|\mathbf{b}_x|^2\big)dx\right\}ds \\
    & \leq  C+C\int_0^t\|\theta\|_{L^\infty(\Omega)}\int_{\Omega}u_x^2
    dxds
    +C\left(\iint_{Q_t}u_{xx}^2 dxds\right)^{1/2}.
  \end{split}
\end{equation*}
Inserting them into \eqref{theta0} yields
\begin{equation}\label{theta33}
\begin{aligned}
     &\int_\Omega  \theta^2dx+ \iint_{Q_t}  \kappa \theta_x^2dxds\\
      &\leq  C +C\int_0^t\left[\|\theta\|_{L^\infty(\Omega)}+\|\theta\|_{L^\infty(\Omega)}^{2}\right]\int_{\Omega}u_x^2  dxds +C\left(\iint_{Q_t}u_{xx}^2 dxds\right)^{1/2}.
  \end{aligned}
\end{equation}
Plugging it into \eqref{0u1} gives
\begin{equation*}\label{0u2}
\begin{split}
&  \int_\Omega u_x^2dx+\iint_{Q_t} \big(u_t^2+ u_{xx}^2\big)dxdt\\
&\leq  C +C\int_0^t\left[\|u^2\|_{L^\infty(\Omega)}+\|\theta\|_{L^\infty(\Omega)}+\|\theta\|_{L^\infty(\Omega)}^{2}\right]
\int_{\Omega}u_x^2  dxds +C\left(\iint_{Q_t}u_{xx}^2 dxds\right)^{1/2}\\
 &\leq  C +C\int_0^t\left[\|u^2\|_{L^\infty(\Omega)}+\|\theta\|_{L^\infty(\Omega)}+\|\theta\|_{L^\infty(\Omega)}^{2}\right]
 \int_{\Omega}u_x^2  dxds+\frac12\iint_{Q_t}u_{xx}^2 dxds.\\
\end{split}
\end{equation*}
By Gronwall's inequality and noticing \eqref{theta00}, we have
\begin{equation*}
\begin{split}
  \int_\Omega u_x^2dx+\iint_{Q_t} \big(u_t^2 + u_{xx}^2\big)dxdt   \leq  C.
\end{split}
\end{equation*}
Consequently,  \eqref{w1112} follows  from \eqref{theta33} and Lemmas 2.7-2.9 .
  The proof is complete.
\end{proof}

As a consequence of Lemma 2.10, we have
\begin{equation}\label{ux4}
\begin{split}
  \iint_{Q_T} u_x^6dxdt\leq& C\int_0^T\|u_x\|_{L^\infty(\Omega)}^4\int_\Omega u_x^2dx dt\leq  C\int_0^T\|u_x\|_{L^\infty(\Omega)}^4 dt\\
    \leq&  C\int_0^T\left(\int_\Omega|u_x||u_{xx}|dx\right)^2 dt\leq
    C\int_0^T\left(\int_\Omega u_x^2dx\right)\left(\int_\Omega u_{xx}^2dx\right)dt\\
    \leq&  C.
\end{split}
\end{equation}

 \begin{lemma}\label{2.11}
  Let \eqref{kappa} and \eqref{assumption1} hold. Then $\|\mathbf{w}\|_{L^\infty(Q_T)}\leq C$. Moreover, it holds
\begin{equation*}\begin{split}
 \sup\limits_{0<t<T}\int_\Omega |\mathbf{w}_x | dx \leq C.
\end{split}\end{equation*}
\end{lemma}
 \begin{proof}
  Set $\mathbf{z}=\mathbf{w}_x$. Differentiating \eqref{w12}  in $x$ gives
\begin{equation}\label{zx}
\begin{split}
\mathbf{z}_t=\left(\frac{\mu}{\rho}\mathbf{z}_x\right)_x-(u\mathbf{z})_x+\left(\frac{\mathbf{b}_x}{\rho}\right)_x.
\end{split}\end{equation}
Denote $\Phi_\epsilon(\cdot):\mathbb{R}^2\rightarrow \mathbb{R}^+ $ for $\epsilon \in (0, 1)$  by
\begin{equation*}\label{phiepsilon}
\Phi_\epsilon(\xi)=\sqrt{\epsilon^2+|\xi|^2},\quad\forall\xi\in\mathbb{R}^2.
\end{equation*}
Observe that $\Phi_\epsilon$ has the properties
 \begin{equation}\label{phii}
\left\{\begin{split}
  &|\xi|\leq |\Phi_\epsilon(\xi)|\leq |\xi|+\epsilon,\quad\forall\xi\in\mathbb{R}^2,\\
  &|\nabla_\xi\Phi_\epsilon(\xi)|\leq 1,\quad\forall\xi\in\mathbb{R}^2,\\
  &   0\leq \xi\cdot\nabla_\xi\Phi_\epsilon(\xi)\leq \Phi_\epsilon(\xi),\quad\forall\xi\in\mathbb{R}^2,\\
  &\eta D_\xi^2\Phi_\epsilon(\xi)\eta^{\top}\geq 0,\quad \forall\xi, \eta\in\mathbb{R}^2,\\
  &\lim\limits_{\epsilon\rightarrow 0^+}\Phi_\epsilon(\xi)=|\xi|,\quad\forall\xi\in\mathbb{R}^2,
\end{split}
\right.
\end{equation}
where $\xi^{\top}$ stands for the transpose of the vector
$\xi=(\xi_1,\xi_2) \in \mathbb{R}^2$, and $D_\xi^2 g$ is the
Hessian matrix of the function $g :\mathbb{R}^2\rightarrow\mathbb{R}$ which is defined by
\begin{equation*}
D^2_\xi g(\xi)=\left(\begin{split}
  &g_{\xi_1\xi_1}&&g_{\xi_1\xi_2}\\
  &g_{\xi_2\xi_1}&&g_{\xi_2\xi_2}
\end{split}
\right).
\end{equation*}
Multiplying \eqref{zx} by $\nabla_\xi\Phi_\epsilon(\mathbf{z})$ and integrating over $Q_t$, we have
\begin{equation}\label{wx1}
\begin{split}
&\int_\Omega \Phi_\epsilon(\mathbf{z})dx-\int_\Omega \Phi_\epsilon(\mathbf{w}_{0x})dx\\
&=-\mu\iint_{Q_t}\frac{1}{\rho} \mathbf{z}_x D_\xi^2\Phi_\epsilon(\mathbf{z})(\mathbf{z}_x)^\perp dxds
-\iint_{Q_t}(u\mathbf{z})_x\cdot\nabla_\xi\Phi_\epsilon(\mathbf{z})  dxds\\
&\quad+\iint_{Q_t}\left(\frac{\mathbf{b}_x}{\rho}\right)_x\cdot\nabla_\xi\Phi_\epsilon(\mathbf{z})  dxds
+\mu\int_0^t\frac{\mathbf{z}_x\cdot\nabla_\xi\Phi_\epsilon(\mathbf{z}) }{\rho}\bigg|_{x=0}^{x=1}ds =:\sum_{j=1}^4E_j.
\end{split}
\end{equation}
From \eqref{phii}$_4$  it follows that
$$
E_1\leq 0.
$$
To estimate $E_2$, we observe by \eqref{phii}$_3$
\begin{equation*}
\begin{split}
E_2=& -\iint_{Q_t}\big(u\mathbf{z}_x+u_x \mathbf{z}\big)\cdot\nabla_\xi\Phi_\epsilon(\mathbf{z})  dxds\\
=& \iint_{Q_t}\big(u_x\Phi_\epsilon(\mathbf{z})-u_x \mathbf{z}\cdot\nabla_\xi\Phi_\epsilon(\mathbf{z}) \big) dxds\\
\leq&  C\int_0^t\|u_x\|_{L^\infty(\Omega)} \int_\Omega\Phi_\epsilon(\mathbf{z})dxds.
\end{split}
\end{equation*}
As to $E_3$, utilizing the equation \eqref{e1}$_4$ yields
\begin{equation*}
\begin{split}
 E_3 =&\iint_{Q_t}\frac{\mathbf{b}_{xx}\cdot\nabla_\xi\Phi_\epsilon(\mathbf{z})}{\rho}dxds
 -\iint_{Q_t}\frac{\mathbf{b}_{x}\cdot\nabla_\xi\Phi_\epsilon(\mathbf{z})}{\rho^2}\rho_xdxds\\[1mm]
 =&\frac{1}{\nu}\iint_{Q_t} \frac{\big[\mathbf{b}_{t}+(u\mathbf{b})_x-\mathbf{z}\big]\cdot\nabla_\xi\Phi_\epsilon(\mathbf{z})}{\rho} dxds -\iint_{Q_t}\frac{\mathbf{b}_{x}\cdot\nabla_\xi\Phi_\epsilon(\mathbf{z}) }{\rho^2}\rho_xdxds\\[1mm]
   \leq&C \iint_{Q_t}  \big[|\mathbf{b}_t|+|(u\mathbf{b})_x|+|\rho_x||\mathbf{b}_x|\big]dxds
 \leq C,
\end{split}
\end{equation*}
where we used \eqref{phii}$_2$-\eqref{phii}$_3$ and  Lemmas \ref{2.4} and \ref{2.10}.

It remains to estimate $E_4$. From \eqref{w12}, we have
\begin{equation}\label{wt}
\begin{split}
  \left|\frac{\mu}{\rho(a,t)}\mathbf{z}_x(a,t)\right|
  =\left|\mathbf{w}_t(a,t)-\frac{\mathbf{b}_x(a,t)}{\rho(a,t)}\right|
  \leq C+C |\mathbf{b}_x(a,t)|,\quad \hbox{\rm where}~a=0~\hbox{\rm
  or}~a=1.
\end{split}
\end{equation}
On the other hand, we first integrate \eqref{e1}$_4$ from $a$ to $y
\in [0, 1]$ in $x$, and then integrate the resulting equation over
$(0, 1)$ in $y$, so that
 \begin{equation*}\begin{split}
\mathbf{b}_x(a,t)=&-\frac{1}{\nu}\left\{\int_0^1\hspace{-2mm}\int_a^y
\mathbf{b}_t(x,t)
dxdy+\int_0^1(u\mathbf{b}-\mathbf{w})(y,t)dy+\mathbf{w}(a,t)\right\},
\end{split}\end{equation*}
so it follows from Lemmas \ref{2.1} and \ref{2.10} that
\begin{equation*}\begin{split}
\int_0^T|\mathbf{b}_x(a,t)|^2dt  \leq C.
\end{split}\end{equation*}
Thus one derives from \eqref{wt} that
\begin{equation*}
\begin{split}
  \int_0^T\left|\frac{\mu}{\rho(a,t)}\mathbf{z}_x(a,t)\right|dt \leq C+C\int_0^T|\mathbf{b}_x(a,t)| dt\leq C,
\end{split}
\end{equation*}
therefore
\begin{equation*}
\begin{split}
  E_4\leq C\int_0^T\left\{\left|\frac{\mu}{\rho(1,t)}\mathbf{z}_x(1,t)\right|
  +\left|\frac{\mu}{\rho(0,t)}\mathbf{z}_x(0,t)\right|\right\}dt  \leq C.
\end{split}
\end{equation*}
Substituting the above results in \eqref{wx1} and utilizing
Gronwall's inequality, we get
$$\int_\Omega \Phi_\epsilon(\mathbf{z})dx\leq C+\int_\Omega \Phi_\epsilon(\mathbf{w}_{0x})dx.
$$
Passing to the limit as $\epsilon\rightarrow 0$ yields
$$\int_\Omega |\mathbf{w}_x|dx\leq C.
$$
This and $\int_\Omega |\mathbf{w}|^2dx\leq C$  imply that
$|\mathbf{w}|\leq C$. The proof is complete.
 \end{proof}

\begin{lemma}\label{2.12}\label{upperbound}Let \eqref{kappa} and \eqref{assumption1} hold. Then
\begin{equation*}\begin{split}
 \int_0^T\|\mathbf{b}_x\|_{L^\infty(\Omega)}^2dt \leq C.
\end{split}\end{equation*}
\end{lemma}
 \begin{proof}
 For any fixed $ z \in [0, 1]$, we first integrate \eqref{e1}$_4$ from $z$ to $y \in [0, 1]$ in $x$, and then integrate the resulting equation over $(0, 1)$ in $y$, so that
 \begin{equation*}\begin{split}
\mathbf{b}_x(z,t)=&-\frac{1}{\nu}\left\{\int_0^1\hspace{-2mm}\int_z^y
\mathbf{b}_t(x,t)
dxdy+\int_0^1(u\mathbf{b}-\mathbf{w})(y,t)dy-(u\mathbf{b}-\mathbf{w})(z,t)\right\},
\end{split}\end{equation*}
  which together with Lemmas \ref{2.10} and \ref{2.11}  implies the desired result. The proof is complete.

 \end{proof}

Combining Lemmas \ref{2.10}-\ref{2.12}, we have
 \begin{equation}\label{bxxbound}\begin{split}
  \iint_{Q_T}|\mathbf{b}_x||\mathbf{b}_{xx}|dxdt
 &=\frac{1}{\nu}\iint_{Q_T}|\mathbf{b}_x||\mathbf{b}_{t}+(u\mathbf{b})_x-\mathbf{w}_x|dxdt\\
  &\leq C+C\int_0^T\|\mathbf{b}_x\|_{L^\infty(\Omega)}\int_\Omega|\mathbf{w}_x|dxdt \leq
 C,
\end{split}\end{equation}
and
\begin{equation}\label{bxbound}
\begin{split}
\iint_{Q_T}|\mathbf{b}_x |^4dxdt\leq C\int_0^T\|\mathbf{b}_x\|_{L^\infty(\Omega)}^2\int_\Omega|\mathbf{b}_x|^2dxdt\leq   C.
\end{split}\end{equation}

Now some results in Lemmas \ref{2.6} and \ref{2.10} can be improved
as follows.
 \begin{lemma}\label{2.13}Let \eqref{kappa} and \eqref{assumption1} hold. Then
\begin{equation*}\label{we1}
\begin{split}
&
\sqrt{\mu}\sup\limits_{0<t<T}\int_\Omega|\mathbf{w}_x|^2dx+\mu^{3/2}
\iint_{Q_T}
 |\mathbf{w}_{xx}|^2 dxdt \leq  C,\\
  & \sup\limits_{0<t<T}\int_\Omega|\mathbf{w}_x|^2\omega  dx+\iint_{Q_T} \big(\mu |\mathbf{w}_{xx}|^2+|\mathbf{b}_{xx}|^2\big)
  \omega dxdt\leq C.
\end{split}
\end{equation*}
 \end{lemma}
\begin{proof}
For the first estimate, we can use an argument similar to  Lemma 2.6
to finish the proof. The key is to deal with the  term
$-\mu\iint_{Q_t}\frac{1}{\rho}\mathbf{b}_x\cdot\mathbf{w}_{xx}dxds $
in \eqref{0v100}.

By integrating by parts and using Cauchy-Schwarz's inequality, we
have
\begin{equation}\label{w0}
\begin{split}
& -\mu\iint_{Q_t}\frac{1}{\rho}\mathbf{b}_x\cdot\mathbf{w}_{xx}dxds\\[1mm]
 &=\mu\iint_{Q_t}\frac{\mathbf{b}_{xx}\cdot\mathbf{w}_{x}}{\rho}dxds
 -\mu\iint_{Q_t}\frac{\mathbf{b}_{x}\cdot\mathbf{w}_{x}}{\rho^2}\rho_xdxds
 -\mu\int_0^T \frac{\mathbf{b}_x\cdot\mathbf{w}_{x}}{\rho}\bigg|_{x=0}^{x=1} ds\\[1mm]
 &\leq C\mu\iint_{Q_t}| \mathbf{b}_{xx}|^2 dxds+C\mu\iint_{Q_t}|\mathbf{w}_{x}|^2dxds+\mu\iint_{Q_t}|\mathbf{b}_{x}|^2 \rho_x^2dxds\\[1mm]
 &\quad+C\mu\left(\int_0^t\|\mathbf{b}_{x}\|_{L^\infty(\Omega)}^2 ds\right)^{1/2}\left(\int_0^t\|\mathbf{w}_{x}\|_{L^\infty(\Omega)}^2 ds\right)^{1/2}\\[1mm]
 &\leq C\mu + C\mu\iint_{Q_t}| \mathbf{b}_{xx}|^2 dxds+C\mu\iint_{Q_t}|\mathbf{w}_{x}|^2dxds
 +C\mu \left(\int_0^t\|\mathbf{w}_{x}\|_{L^\infty(\Omega)}^2 ds\right)^{1/2},\\
  \end{split}
\end{equation}
where we used the fact  by Lemmas \ref{2.4} and \ref{2.12}
\begin{equation*}
\begin{split}
  \iint_{Q_t}|\mathbf{b}_{x}|^2|\rho_x|^2dxds \leq
  \int_0^t\|\mathbf{b}_{x}\|_{L^\infty(\Omega)}^2\int_\Omega\rho_x^2dxds\leq
  C.
  \end{split}
\end{equation*}
By \eqref{wx2}, we obtain
\begin{equation}\label{w01}
\begin{split}
  &\mu \left(\int_0^t\|\mathbf{w}_{x}\|_{L^\infty(\Omega)}^2 ds\right)^{1/2}\\
   &\leq C\mu+C \mu^{1/4}\left(\mu\iint_{Q_t}|\mathbf{w}_{x}|^2dxds\right)^{1/4}
   \left(\mu^2\iint_{Q_t}|\mathbf{w}_{xx}|^2dxds\right)^{1/4}\\[1mm]
&\leq C\sqrt{\mu}+\frac{C  \mu}{\epsilon}\iint_{Q_t}
|\mathbf{w}_{x}|^2dxds+\epsilon \mu^2
\iint_{Q_t}\frac{1}{\rho}|\mathbf{w}_{xx}|^2dxds,~ \forall
\epsilon>0.
\end{split}
\end{equation}
It remains to show the estimate
\begin{equation}\label{bxx}
\begin{split}
 & \iint_{Q_t}|\mathbf{b}_{xx}|^2dxds \leq C+ C\iint_{Q_t}  |\mathbf{w}_x|^2dxds.
\end{split}
\end{equation}
Multiplying \eqref{e1}$_4$ by $\mathbf{b}_{xx}$ and integrating over
$Q_t$, we have
\begin{equation*}
\begin{split}
 &\frac12\int_\Omega|\mathbf{b}_x|^2dx+\nu\iint_{Q_t}|\mathbf{b}_{xx}|^2dxds\\
   &=\frac12\int_\Omega|\mathbf{b}_{0x}|^2dx+\iint_{Q_t}u_x\mathbf{b}\cdot\mathbf{b}_{xx} dxds-\frac12\iint_{Q_t}u_x|\mathbf{b}_x|^2 dxds-\iint_{Q_t}  \mathbf{w}_x\cdot\mathbf{b}_{xx}dxds \\
 &\leq C+\frac{\nu}{2}\iint_{Q_t}  |\mathbf{b}_{xx}|^2 dxds+C\iint_{Q_t}  |\mathbf{w}_x|^2 dxds +C\int_0^t\|u_x\|_{L^\infty(\Omega)}\int_{\Omega}|\mathbf{b}_x|^2 dxds,
 \end{split}
\end{equation*}
where we used Lemma \ref{2.10}. Thus, \eqref{bxx} follows from Gronwall's inequality.

Inserting  the above estimates into \eqref{w0} and taking a small
$\epsilon>0$, we have
\begin{equation*}
\begin{split}
  -\mu\iint_{Q_t}\frac{1}{\rho}\mathbf{b}_x\cdot\mathbf{w}_{xx}dxdt \leq C\sqrt{\mu} +C\mu\iint_{Q_t}|\mathbf{w}_{x}|^2dxdt+\frac{\mu^2}{4}
\iint_{Q_t}\frac{1}{\rho}|\mathbf{w}_{xx}|^2dxds.
  \end{split}
\end{equation*}
Then, an argument similar to Lemma \ref{2.6} leads to
\begin{equation*}
\begin{split}
  &\mu\int_\Omega|\mathbf{w}_x|^2dx
 +\mu^2\iint_{Q_t} |\mathbf{w}_{xx}|^2dxds  \\
 &\leq C\sqrt{\mu}
 +C\int_0^t\big(1+\|u_x\|_{L^\infty(\Omega)}\big)\left(\mu\int_{\Omega}|\mathbf{w}_x|^2dx\right)ds.
  \end{split}
\end{equation*}
So   the first estimate of this lemma follows from Gronwall's inequality and
\eqref{uxx}.

The second estimate can proved by the arguments similar to Lemma \ref{2.7} and
\eqref{w1}$_1$ and in terms of the first estimate and Lemmas
\ref{2.10}-\ref{2.12}. In fact, this can be done  by using $\omega$ instead of
$\omega^2$ in \eqref{b111} and   \eqref{0v10} and noticing the following facts:
\begin{equation*} \begin{split}
  & \mu\iint_{Q_T}|\mathbf{w}_x\cdot\mathbf{w}_{xx}|dxdt \leq C\sqrt{\mu}\iint_{Q_T}|\mathbf{w}_x|^2dxdt+C\mu^{3/2}\iint_{Q_T}|\mathbf{w}_{xx}|^2dxdt\leq C,\\[1mm]
  & \iint_{Q_T}|\mathbf{b}_x\cdot\mathbf{w}_{x}|dxdt  \leq C\int_0^T\|\mathbf{b}_x\|_{L^\infty(\Omega)}\int_\Omega|\mathbf{w}_x|dxdt \leq
 C.
\end{split}\end{equation*}
 The proof is complete.
\end{proof}

As a consequence of Lemma \ref{2.13} and \eqref{w01}, we also have
\begin{equation}\label{wx4}
\begin{split}
   \mu^{3/2}\iint_{Q_T} |\mathbf{w}_{x}|^4dxdt  \leq C \mu \int_0^T \|\mathbf{w}_{x}\|_{L^\infty(\Omega)}^2 \left(\sqrt{ \mu}\int_{Q_T}|\mathbf{w}_x|^2dx\right)dt\leq C.
  \end{split}
\end{equation}

Based on the above lemmas, we can bound the temperature $\theta$ in a direct way.

\begin{lemma}\label{2.14}
 Let \eqref{kappa} and \eqref{assumption1} hold. Then  $\theta \leq C.$
\end{lemma}
\begin{proof}
 Rewrite the equation \eqref{e1}$_4$  in the form
\begin{equation}\label{theta9}
 \begin{split}
 \theta_t= a(x,t)\theta_{xx}+b(x,t)\theta_x+c(x,t) \theta+f(x,t),
\end{split}
\end{equation}
where
\begin{equation*}
 \begin{split}
a =\rho^{-1}\kappa,\quad b =\rho^{-1}\kappa_x-u,\quad c =-\gamma
u_x,\quad f =\rho^{-1} (\lambda
u_x^2+\mu|\mathbf{w}_x|^2+\nu|\mathbf{b}_x|^2).
\end{split}
\end{equation*}
Set $z=\theta_x$. Differentiating the equation \eqref{theta9} in $x$ yields
\begin{equation}\label{z1}
 \begin{split}
 z_t=(a z_x)_x+(b z)_x+c z+c_x\theta+f_x.
\end{split}
\end{equation}
For $\epsilon \in (0, 1)$, denote $\varphi_\epsilon:\mathbb{R}\rightarrow \mathbb{R}^+$   by
$
\varphi_\epsilon(s)=\sqrt{s^2+\epsilon^2}.
$
Simple calculations show that
\begin{equation*}\label{phi}
 \left\{\begin{split}
 &\varphi_\epsilon'(0)=0,\quad|\varphi_\epsilon'(s)|\leq 1,  \quad\varphi_\epsilon''(s)\geq 0, \quad |s\varphi_\epsilon''(s)|\leq 1,\\
 &\lim\limits_{\epsilon\rightarrow 0} \varphi_\epsilon(s)=|s|, \quad\lim\limits_{\epsilon\rightarrow 0}s\varphi_\epsilon''(s)=0.
\end{split}\right.
\end{equation*}
Multiplying  \eqref{z1} by $\varphi_\epsilon'(z)$, integrating over
$Q_t$, and noticing
$\varphi_\epsilon'(z)|_{x=0,1}=\varphi_\epsilon'(\theta_x)|_{x=0,1}=0$,
we have
\begin{equation*}
 \begin{split}
 \int_\Omega \varphi_\epsilon(z)dx-\int_\Omega \varphi_\epsilon(\theta_{0x})dxds&=-\iint_{Q_t} a \varphi_\epsilon''(z)z_x^2dxds-\iint_{Q_t} b z z_x\varphi_\epsilon''(z)dxds\\
 &\quad+\iint_{Q_t}  (c z+c_x \theta+f_x)\varphi_\epsilon'(z)dxds,\\
 \end{split}
\end{equation*}
and then, we obtain by $\varphi_\epsilon''(s)\geq 0$ and $|\varphi_\epsilon'(s)|\leq 1$
\begin{equation}\label{z2}
 \begin{split}
  &\int_\Omega \varphi_\epsilon(z)dx-\int_\Omega \varphi_\epsilon(\theta_{0x})dx \\
  &\leq \iint_{Q_t}|b z_x||z\varphi_\epsilon''(z)|dxds  + \iint_{Q_t}(|c z|+|c_x \theta|+|f_x|)dxds.
\end{split}
\end{equation}
Recalling $|s\varphi_\epsilon''(s)|\leq 1$ and
$s\varphi_\epsilon''(s)\rightarrow 0$ as $\epsilon\rightarrow 0$ and
using Lebesgue's dominated convergence theorem, we obtain
 \begin{equation*}
 \begin{split}
 \lim\limits_{\epsilon\rightarrow 0}\iint_{Q_T} |b z_x||z\varphi_\epsilon''(z)|dxdt=0.
\end{split}
\end{equation*}
Thus, passing to the limit  as $\epsilon\rightarrow 0$ in \eqref{z2} and using $\lim\limits_{\epsilon\rightarrow 0} \varphi_\epsilon(s)=|s|$, we have
\begin{equation}\label{theta10}
 \begin{split}
  \int_\Omega |\theta_x|dx\leq \int_\Omega|\theta_{0x}|dx  + \iint_{Q_t}(|c \theta_x|+|c_x \theta|+|f_x|)dxds.
\end{split}
\end{equation}
By Lemma \ref{2.10}, we have
\begin{equation*}
 \begin{split}
 &\iint_{Q_T}|c_x\theta| dxdt\leq C\left(\iint_{Q_T}u_{xx}^2dxdt\right)^{1/2}\left(\iint_{Q_T}\theta^2dxdt\right)^{1/2}\leq C,\\
&\iint_{Q_T}|c \theta_x| dxdt\leq C\left(\iint_{Q_T}u_{x}^2dxdt\right)^{1/2}\left(\iint_{Q_T}\theta_x^2dxdt\right)^{1/2}\leq C.
\end{split}
\end{equation*}
By Cauchy-Schwarz's inequality,  \eqref{bxxbound}, Lemmas \ref{2.4}
and \ref{2.10}, \eqref{ux4}, \eqref{bxbound} and \eqref{wx4}, we
obtain
\begin{equation*}\label{theta6}
 \begin{split}
 \iint_{Q_T}|f_x| dxdt
&\leq C\iint_{Q_T}\big(|u_x||u_{xx}|+\mu|\mathbf{w}_x\cdot\mathbf{w}_{xx}|+ |\mathbf{b}_x\cdot\mathbf{b}_{xx}|\big)dxdt\\
&\quad+C\iint_{Q_T}\big(u_x^2+\mu|\mathbf{w}_x|^2+ |\mathbf{b}_x|^2\big)|\rho_x|dxdt\\
&\leq C+C\iint_{Q_T}\big(u_x^2+u_{xx}^2 + \sqrt{\mu} |\mathbf{w}_x|^2+\mu^{3/2} |\mathbf{w}_{xx}|^2\big)dxdt \\
&\quad+C\iint_{Q_T}\big(u_x^4+\mu^2|\mathbf{w}_x|^4+|\mathbf{b}_x|^4\big)dxdt+C\iint_{Q_T}\rho_x^2dxdt \leq C.
\end{split}
\end{equation*}
 Substituting the above estimates into \eqref{theta10} yields
$$
\int_\Omega|\theta_x|dx\leq C,
$$
which together with $\int_\Omega\theta dx\leq C$  implies the desired result. The proof is complete.
\end{proof}

By means of the bounds of $\theta$, we can obtain easily the
following estimates.

\begin{lemma}\label{2.15}
Let \eqref{kappa} and \eqref{assumption1} hold. Then
\begin{equation*}\begin{split}
\sup\limits_{0<t<T}\int_\Omega \theta_x^2dx+\iint_{Q_T} \big(\theta_t^2+\theta_{xx}^2\big)dxdt\leq C.
\end{split}\end{equation*}
\end{lemma}
\begin{proof} Rewrite the equation \eqref{e1}$_5$ in the form
\begin{equation}\label{theta1}
\begin{split}
 \rho\theta_t-(\kappa\theta_x)_x=\mathcal{Q}-\rho
 u\theta_x-\gamma\rho\theta u_x:=f.
\end{split}\end{equation}
We first estimate  $\|f\|_{L^2(Q_T)}$. By  \eqref{ux4},
\eqref{bxbound}, \eqref{wx4}  and Lemmas \ref{2.10} and \ref{2.14},
we have
\begin{equation}\label{f}
\begin{split}
 \iint_{Q_T}f^2dxdt\leq &
   C\iint_{Q_T}(u_x^4+\mu^2|\mathbf{w}_x|^4+\nu^2|\mathbf{b}_x|^4+\rho^2u^2\theta_x^2+\rho^2u_x^2\theta^2)dxdt \leq C.
\end{split}\end{equation}
 Multiplying  \eqref{theta1} by $\kappa\theta_t$ and  integrating over $Q_t$, we have
\begin{equation}\label{theta2}
\begin{split}
 &\iint_{Q_t} \rho\kappa\theta_t^2 dxdt+\iint_{Q_t}  \kappa\theta_x (\kappa\theta_t)_x dxdt
 =\iint_{Q_t}f\kappa\theta_tdxdt.
\end{split}\end{equation}
Observe that
$$
(\kappa\theta_t)_x=(\kappa\theta_x)_t+\kappa_\rho\rho_x\theta_t+\kappa_\rho\theta_x(\rho_xu+\rho u_x),
$$
so that
\begin{equation*}\label{theta3}
\begin{split}
  \iint_{Q_t}  \kappa\theta_x (\kappa\theta_t)_x dxdt=& \frac12\int_\Omega \kappa^2\theta_x^2dx-\frac12\int_\Omega \kappa^2(\rho_0,\theta_0)\theta_{0x}^2dx\\
&+\iint_{Q_t}\Big[\kappa\kappa_\rho\rho_x\theta_x\theta_t+\kappa\kappa_\rho\theta_x^2(\rho_xu+\rho
u_x)\Big]dxdt,
\end{split}\end{equation*}
and substitute it into \eqref{theta2} to yield
\begin{equation}\label{theta8}
\begin{split}
 &\iint_{Q_t} \rho\kappa\theta_t^2 dxdt+\int_\Omega
 \kappa^2\theta_x^2dx\\
& \leq C-2\iint_{Q_t}\Big[\kappa\kappa_\rho\rho_x\theta_x\theta_t
 +\kappa\kappa_\rho\theta_x^2(\rho_xu+\rho u_x)-f\kappa\theta_t\Big]dxdt.\\
\end{split}\end{equation}
By the estimates $C^{-1}\leq \rho, \theta \leq C$ and \eqref{kappa},
we have $ \kappa_1\leq \kappa \leq C, |\kappa_\rho| \leq C. $ By
Young's inequality, \eqref{rho}, \eqref{f} and Lemma \ref{2.10}, we
obtain
\begin{equation}\label{theta4}
\begin{split}
 &-2\iint_{Q_t}\Big[\kappa\kappa_\rho\rho_x\theta_x\theta_t+\kappa\kappa_\rho\theta_x^2(\rho_xu+\rho u_x)
 -f\kappa\theta_t\Big]dxdt\\
 & \leq C+\frac14 \iint_{Q_t} \rho\kappa\theta_t^2 dxdt+C\iint_{Q_t} (\kappa\theta_x)^2(\rho_x^2+|\rho_x|+|u_x|) dxds\\
 & \leq C+\frac14 \iint_{Q_t} \rho\kappa\theta_t^2 dxdt+C\int_0^t\|\kappa\theta_x\|_{L^\infty(\Omega)}^2 ds.
\end{split}\end{equation}
Now we are ready to deal with the second integral on right-hand side
of \eqref{theta4}. By the embedding  $W^{1,1}(\Omega)\hookrightarrow
L^\infty(\Omega)$  and Young's inequality, we have
\begin{equation*}\begin{split}
 \int_0^t\|\kappa\theta_x\|_{L^\infty(\Omega)}^2 ds \leq& \iint_{Q_t} |\kappa\theta_x|^2 dxds+2\iint_{Q_t} | \kappa\theta_x ||(\kappa\theta_x)_x| dxds\\
 \leq & \frac{C}{\epsilon}+\frac{\epsilon}{2}\iint_{Q_t} \big|(\kappa\theta_x)_x\big|^2 dxds, \quad \forall \epsilon>0,
\end{split}\end{equation*}
 which together with \eqref{theta1} gives
\begin{equation*}\begin{split}
 \int_0^t\|\kappa\theta_x\|_{L^\infty(\Omega)}^2 ds \leq  \frac{C}{\epsilon}
 + \epsilon\iint_{Q_t} (\rho^2\theta_t^2+f^2)dxdt.
\end{split}\end{equation*}
Plugging it into \eqref{theta4}, taking a small $\epsilon>0$ and
using \eqref{f}, we obtain
\begin{equation*}
\begin{split}
  -2\iint_{Q_t}\Big[\kappa\kappa_\rho\rho_x\theta_x\theta_t+\kappa\kappa_\rho\theta_x^2(\rho_xu+\rho u_x)
  -f\kappa\theta_t\Big]dxdt
   \leq C+\frac12 \iint_{Q_t} \rho\kappa\theta_t^2 dxdt,
\end{split}\end{equation*}
from which, \eqref{kappa} and \eqref{theta8} it follows that
\begin{equation}\label{theta}
\begin{split}
\sup\limits_{0<t<T}\int_\Omega \theta_x^2dx+\iint_{Q_T}  \theta_t^2
dxdt\leq C.
\end{split}\end{equation}

By \eqref{theta}  and Lemma \ref{2.14}, one can derive easily from \eqref{e1}$_5$ that $\|\theta_{xx}\|_{L^2(Q_T)} \leq C$. The proof is complete.
\end{proof}

Due to Lemma \ref{2.15}, an argument similar to \eqref{ux4} yields
\begin{equation}\label{theta5}
\begin{split}
\iint_{Q_T}\theta_x^6dxdt\leq   C.
\end{split}\end{equation}

 Thus, all the estimates appearing in Theorem 1.1 are proved.

\subsection{Proof of Theorem 1.1(ii)}
 By an argument similar to \eqref{rho}$_2$, one has
\begin{equation}\label{base5}
\begin{aligned}
&\|(u, \mathbf{b}, \theta)\|_{C^{1/2, 1/4}(\overline Q_T)}\leq C,\\
&\|\mathbf{w}\|_{C^{1/2, 1/4}([\delta, 1-\delta]\times[0, T])}\leq C,~\forall\delta\in\big(0, (b-a)/2\big) .
  \end{aligned}
  \end{equation}
From  \eqref{rho}$_2$,  \eqref{ux4}, \eqref{bxbound},
\eqref{theta5}, \eqref{base5} and Lemmas \ref{2.2}, \ref{2.4},
\ref{2.10}-\ref{2.15} it follows that there exist a subsequence
$\mu_j \rightarrow 0$  and  $(\overline\rho, \overline u, \overline{\mathbf{w}},
\overline{\mathbf{b}}, \overline\theta) \in \mathbb{F}$
such that  the corresponding  solution
for problem \eqref{e1}-\eqref{e4} with $\mu=\mu_j $, still denoted by $(\rho,u,\mathbf{w},\mathbf{b}, \theta)$, converges in the sense:
\begin{equation*}\label{rate}
 \begin{split}
 &(\rho,u,\mathbf{b}, \theta)\rightarrow (\overline\rho,\overline u,\overline{\mathbf{b}}, \overline\theta)~~
 \hbox{\rm strongly in}~~C^\alpha(\overline Q_T),~\forall\alpha\in(0, 1/4),\\[1mm]
 &(\rho_t,\rho_x,u_x,\mathbf{b}_x,\theta_x)\rightharpoonup(\overline\rho_t, \overline\rho_x,\overline u_x,
 \overline{\mathbf{b}}_x,\overline\theta_x)~~ \hbox{\rm weakly}-*~\hbox{\rm in}~ L^\infty(0, T; L^2(\Omega)),\\[1mm]
  &(u_t,\mathbf{b}_t,\theta_t,u_{xx}, \theta_{xx})\rightharpoonup(\overline u_t,\overline{\mathbf{b}}_t,
  \overline\theta_t,\overline u_{xx},  \overline\theta_{xx})~~ \hbox{\rm  weakly in}~~L^2(Q_T),\\[1mm]
  & \mathbf{b}_{xx}  \rightharpoonup  \overline{\mathbf{b}}_{xx}~~ \hbox{\rm weakly in}~~L^2((a+\delta,b-\delta)\times(0, T)),~\forall \delta\in(0, (b-a)/2),\\
      \end{split}
  \end{equation*}
 and
  \begin{equation*}
   \begin{split}
 & \mathbf{w}  \rightarrow   \overline{\mathbf{w}}\quad\hbox{\rm strongly  in}~~C^\alpha([a+\delta,b-\delta]\times[0, T]),~\forall \delta\in\big(0, (b-a)/2\big),~\alpha\in(0, 1/4),\\
 & \mathbf{w}_t  \rightharpoonup   \overline{\mathbf{w}}_t   \quad\hbox{\rm weakly in}~~L^2(Q_T),\\
   &   \mathbf{w}_x \rightharpoonup  \overline{\mathbf{w}}_x\quad\hbox{\rm weakly}-*~ \hbox{\rm in}~ L^\infty(0, T; L^2(a+\delta,b-\delta)),~\forall \delta\in(0, (b-a)/2),\\
 &\mathbf{w}\rightarrow  \overline{\mathbf{w}}~~ \hbox{\rm strongly in}~~L^r(Q_T),\quad\forall r \in [1, +\infty),\\
 &\sqrt{\mu}\|\mathbf{w}_x\|_{L^4(Q_T)} \rightarrow 0.
   \end{split}
    \end{equation*}

    Next we show the strong convergence of $(u_x,\mathbf{b}_x,\theta_x)$ in $L^2(Q_T)$. Multiplying \eqref{e1}$_2$ with $\mu=\mu_j$ by $(u-\overline u)$ and integrating over $Q_T$, we have
\begin{equation*}
\begin{split}
  &\lambda\iint_{Q_T} \big(u_{x}-\overline u_x\big)^2dxdt+\lambda\iint_{Q_T}\overline u_x\big(u_{x}-\overline u_x\big)dxdt\\
  &=-\iint_{Q_T}\left[(\rho u)_{t}
  +\left(\rho u^2+\gamma\rho\theta +\frac12|\mathbf{b}|^2\right)_x\right](u-\overline u)   dxdt,
 \end{split}
\end{equation*}
which together with Lemmas \ref{2.4}, \ref{2.10}  and  \ref{2.14} implies that
$$
u_{x}\rightarrow\overline u_x~\hbox{\rm strongly in}~ L^2(Q_T) ~\hbox{\rm as}~\mu_j\rightarrow 0.
$$
Similarly, one has
\begin{equation*}
\begin{split}
&(\mathbf{b}_{x},\theta_{x})\rightarrow (\overline{\mathbf{b}}_x,\overline \theta_x)~\hbox{\rm strongly  in}~ L^2(Q_T) ~\hbox{\rm as}~\mu_j\rightarrow 0.
 \end{split}
\end{equation*}
Furthermore, since  from  \eqref{ux4}, \eqref{bxbound} and
\eqref{theta5}, we have
\begin{equation*}
\begin{split}
&(u_x, \theta_x)\rightarrow(\overline u_x,\overline\theta_x)~~\hbox{\rm strongly  in}~ L^{s_1}(Q_T)
 ~\hbox{\rm as}~\mu_j\rightarrow 0,~\forall s_1 \in [1, 6),\\
 &\mathbf{b}_{x}\rightarrow \overline{\mathbf{b}}_x ~~\hbox{\rm strongly  in}~ L^{s_2}(Q_T)
 ~\hbox{\rm as}~\mu_j\rightarrow 0,~\forall s_2 \in [1, 4).\\
  \end{split}
\end{equation*}

Then, it is easy to check that  $(\overline\rho,\overline u,
\overline{\mathbf{w}}, \overline{\mathbf{b}}, \overline\theta)$
  satisfies \eqref{equations}.

   On the other hand, one can see from Theorem 1.1(iii) that the limit problem
\eqref{equations} admits at most one solution in  $\mathbb{F}$. Thus, the above convergence relations hold for any $\mu_j\rightarrow 0$. The proof of Theorem 1.1(ii) is then completed.

\subsection{Proof of Theorem 1.1(iii)}
The proof is divided into several steps among which the fourth step is the key that can be proved in terms of the boundary estimates of $\mathbf{w}_x$. For convenience, we set
\begin{equation*}
\begin{split}
&\widetilde{\rho}=\rho-\overline\rho,\quad \widetilde{u}= u-\overline
u,\quad\widetilde{\mathbf{w}}=\mathbf{w}-\overline{\mathbf{w}},
\quad\widetilde{\mathbf{b}}=\mathbf{b}-\overline{\mathbf{b}},\quad
\widetilde{\theta}=\theta-\overline\theta,\\[1mm]
&\mathbb{H}(t)=\|(\widetilde{\rho}, \widetilde{u},\widetilde{\mathbf{w}}, \widetilde{\mathbf{b}}, \widetilde{\theta})\|_{L^2(\Omega)}^2,\\[1mm]
&D(t)=1+ \|(u_x,\mathbf{b}_x, \overline u_x, \overline{\mathbf{b}}_x,\overline{\theta}_x)\|_{L^\infty(\Omega)}^2+\|(\overline u_t, \overline\theta_t, \overline u_x, \overline{\mathbf{b}}_x,\overline{\theta}_x)\|_{L^2(\Omega)}^2.
 \end{split}
\end{equation*}
Clearly, $D(t)\in L^1(0, T)$.

\indent{\bf Step 1} We claim that
\begin{equation}\label{rho4}
\begin{split}
  \int_\Omega \widetilde{\rho}^2dx\leq  \epsilon\iint_{Q_t} \widetilde{u}_x^2dxds+\frac{C}{\epsilon}\int_0^tD(s)\mathbb{H}(s)ds,~~ \forall \epsilon \in (0, 1).
  \end{split}
\end{equation}From \eqref{e1}$_1$ and \eqref{equations}$_1$ it follows that
$$
\widetilde{\rho}_t=-\big(\rho \widetilde{u} +\overline
u\widetilde{\rho}\big)_x.
$$
Multiplying it by $\widetilde{\rho}$ and integrating over $Q_t$, we
have by Young's inequality
\begin{equation*}\label{rho1}
\begin{split}
  \frac12\int_\Omega \widetilde{\rho}^2dx=&-\iint_{Q_t}\big(\rho \widetilde{u}_x\widetilde{\rho}+\rho_x\widetilde{u}\widetilde{\rho} \big)dxds -\frac12\iint_{Q_t}\overline
  u_x\widetilde{\rho}^2dxds\\
    \leq&\frac{\epsilon}{4}\iint_{Q_t} \widetilde{u}_x^2dxds+C\iint_{Q_t} \widetilde{u}^2\rho_x^2dxds\\
    & +\frac{C}{\epsilon}
    \int_0^t(1+\|\overline u_x\|_{L^\infty(\Omega)})\int_\Omega\widetilde{\rho}^2dxds, \forall \epsilon \in (0,
    1).
  \end{split}
\end{equation*}
Since from \eqref{rho}, we have
\begin{equation*}\label{rho3}
\begin{split}
   C\iint_{Q_t} \widetilde{u}^2\rho_x^2dxds
   \leq C\int_0^t\|\widetilde{u}\|_{L^\infty(\Omega)}^2 ds\leq \frac{\epsilon}{4}\iint_{Q_t} \widetilde{u}_x^2dxds+\frac{C}{\epsilon}\iint_{Q_t} \widetilde{u}^2dxds.
  \end{split}
\end{equation*}
Thus, the claim \eqref{rho4} is proved.

\indent{\bf Step 2}  We claim that
\begin{equation}\label{u14}
\begin{split}
   \int_\Omega  \widetilde{u}^2dx+ \iint_{Q_t} \widetilde{u}_x ^2dxds\leq   C \int_0^tD(s)\mathbb{H}(s)ds.
\end{split}
\end{equation}
Using
\eqref{e1}$_1$ and \eqref{equations}$_1$, we derive from \eqref{e1}$_2$ and \eqref{equations}$_2$ that
  \begin{equation*}
\begin{split}
 &\big(\rho \widetilde{u} \big)_t+\big(\rho u\widetilde{u}\big)_x+\widetilde{\rho}\overline u_t
 +(\rho u-\overline\rho~\overline u)\overline u_x +\gamma(\rho\theta-\overline\rho\overline\theta)_x
 +\frac12(|\mathbf{b}|^2 -|\overline{\mathbf{b}}|^2 )_x =\lambda \widetilde{u}_{xx}.
  \end{split}
  \end{equation*}
Multiplying it by $\widetilde{u}$ and integrating  over $Q_t$, we have
\begin{equation}\label{uu}
\begin{split}
  &\frac12\int_\Omega\rho \widetilde{u}^2dx+\lambda\iint_{Q_t} \widetilde{u}_x ^2dxds\\
  &=-\iint_{Q_t}\widetilde{\rho} \overline u_t \widetilde{u}dxds
 -\iint_{Q_t}(\rho u-\overline\rho~\overline u)\overline u_x\widetilde{u} dxds
 +\gamma\iint_{Q_t} (\rho\theta-\overline\rho\overline\theta) \widetilde{u}_x dxds
 \\
 &\quad+\frac12\iint_{Q_t} (|\mathbf{b}|^2 -|\overline{\mathbf{b}}|^2 ) \widetilde{u}_x dxds   =:\sum_{i=1}^4 I_i.
\end{split}
\end{equation}
Observe that $\rho u-\overline\rho~\overline u=\rho \widetilde{u} +\overline u \widetilde{\rho} $ and
$\rho \theta-\overline\rho\overline\theta =\rho\widetilde{\theta}+\overline\theta \widetilde{\rho} $. We have
\begin{equation}\label{u1}
\begin{split}
  &I_1+I_2 \\
  &\leq  C \int_0^t\|\widetilde{u}\|_{L^\infty(\Omega)} \int_\Omega|\widetilde{\rho}|(|\overline u_t|+|\overline u_x|) dt
  +C\iint_{Q_t} |\overline u_x| \widetilde{u}^2dxds\\
 & \leq  C \int_0^t\left(\int_\Omega\widetilde{\rho}^2 dx\right)^{1/2}
  \left(\int_\Omega(\overline u_t^2+\overline u_x^2)dx\right)^{1/2}  \|\widetilde{u}\|_{L^\infty(\Omega)} dt +C\int_0^t\|\overline u_x\|_{L^\infty(\Omega)}\int_\Omega \widetilde{u}^2dxds\\
    &\leq C \int_0^t \left(\int_\Omega(\overline u_t^2+\overline u_x^2)dx\right)\left(\int_\Omega\widetilde{\rho}^2 dx\right)dt
    +C\int_0^t \|\widetilde{u}\|_{L^\infty(\Omega)}^2 ds\\
    &\quad +C\int_0^t\|\overline u_x\|_{L^\infty(\Omega)}\int_\Omega \widetilde{u}^2dxds \leq C \int_0^tD(s)\mathbb{H}(s)ds
    + \frac\lambda4\iint_{Q_t}   \widetilde{u}_x ^2 dxds,
\end{split}
\end{equation}
and
\begin{equation}\label{u2}
\begin{split}
  I_3&\leq  \frac\lambda4\iint_{Q_t}   \widetilde{u}_x ^2 dxds+C\iint_{Q_t}  (\widetilde{\theta}^2+\widetilde{\rho}^2) dxds.
\end{split}
\end{equation}
Utilizing the estimates
$\|(\mathbf{b},\overline{\mathbf{b}})\|_{L^\infty(Q_T)}\leq C$, we
have
\begin{equation}\label{b01}
\begin{split}
  I_4&\leq  \frac\lambda4\iint_{Q_t}   \widetilde{u}_x ^2 dxds+C\iint_{Q_t} |\widetilde{\mathbf{b}}|^2 dxds.
\end{split}
\end{equation}
Substituting \eqref{u1}-\eqref{b01} into \eqref{uu} completes the proof to \eqref{u14}.

\indent{\bf Step 3} We claim that
\begin{equation}\label{Theta}
\begin{split}
 &\int_\Omega\widetilde{\theta}^2dx+ \iint_{Q_t} \widetilde{\theta}_x^2dxds\\
 &\leq C\sqrt{\mu}
 +\epsilon\iint_{Q_t}\big(\widetilde{u}_x^2+|\widetilde{\mathbf{b}}_x|^2\big)dxds
 +\frac{C}{\epsilon}\int_0^tD(s)\mathbb{H}(s)ds,~~\forall\epsilon\in(0, 1).
\end{split}
\end{equation}
From \eqref{e1}$_5$ and
\eqref{equations}$_5$ it follows that
 \begin{equation*}\label{ll}
\begin{split}
& \big(\rho\widetilde{\theta}\big)_t+(\rho u\widetilde{\theta})_x+\widetilde{\rho}\overline\theta_t+(\rho \widetilde{u}+\overline u \widetilde{\rho})\overline\theta_x+\gamma\rho\theta \widetilde{u}_x
 +\gamma\big(\rho\widetilde{\theta}+\widetilde{\rho}\overline\theta\big)\overline u_x =\big[\kappa(\rho,\theta)\widetilde{\theta}_x\big]_x\\
 &+\big[(\kappa(\rho,\theta)-\kappa(\overline\rho,\overline\theta))\overline\theta_x\big]_x +\lambda(u_x^2-\overline u_x^2) +\mu|\mathbf{w}_x|^2+\nu(|\mathbf{b}_x|^2 -|\overline{\mathbf{b}}_x|^2).
  \end{split}
  \end{equation*}
 Multiplying it by $\widetilde{\theta}$ and integrating over $Q_t$, we obtain
  \begin{equation}\label{21}
\begin{split}
 & \frac12\int_\Omega\rho\widetilde{\theta}^2dx+ \iint_{Q_t}\kappa\widetilde{\theta}_x^2dxds \\
 &=-\iint_{Q_t}\widetilde{\rho}\widetilde{\theta}\overline\theta_tdxdt-\iint_{Q_t}(\rho \widetilde{u}+\overline u \widetilde{\rho})\widetilde{\theta}\overline\theta_xdxds
 -\gamma\iint_{Q_t} \rho\theta \widetilde{u}_x \widetilde{\theta} dxds\\
&\quad-\gamma\iint_{Q_t}\rho \overline u_x\widetilde{\theta}^2dxds -\gamma\iint_{Q_t}\overline\theta\overline u_x\widetilde{\rho}\widetilde{\theta} dxds
-\iint_{Q_t}\overline\theta_x[\kappa(\rho,\theta)-\kappa(\overline\rho,\overline\theta)]\widetilde{\theta}_xdxds \\
&\quad+\lambda\iint_{Q_t} (u_x+\overline u_x)\widetilde{u}_x \widetilde{\theta} dxds+\mu\iint_{Q_t} |\mathbf{w}_x|^2 \widetilde{\theta} dxds\\
&\quad+\nu\iint_{Q_t}(|\mathbf{b}_x|^2 -|\overline{\mathbf{b}}_x|^2)\widetilde{\theta} dxds=:\sum_{i=1}^9 E_i.
\end{split}
\end{equation}
  By H\"{o}lder's inequality and Young's inequality, we have
\begin{equation*}\label{u8}
\begin{split}
 &E_1+E_2+E_5 \leq  C\int_0^t\left(\int_\Omega\big(\widetilde{\rho}^2+ \widetilde{u}^2\big)dx\right)^{1/2}\left(\int_\Omega\widetilde{\theta}^2(\overline\theta_t^2+\overline\theta_x^2+\overline u_x^2)dx\right)^{1/2}dt\\
 &\leq  C\int_0^t \left(\int_\Omega(\overline\theta_t^2+\overline\theta_x^2+\overline u_x^2)dx\right)^{1/2} \left(\int_\Omega(\widetilde{\rho}^2+ \widetilde{u}^2)dx\right)^{1/2}\|\widetilde{\theta}\|_{L^\infty(\Omega)}dt\\
 &\leq  C\int_0^t\left(\int_\Omega(\overline\theta_t^2+\overline\theta_x^2+\overline u_x^2)dx\right)\left(\int_\Omega\big(\widetilde{\rho}^2+ \widetilde{u}^2\big)dx\right)dt
 +\int_0^t\|\widetilde{\theta}\|_{L^\infty(\Omega)}^2ds\\
 &\leq   C\int_0^tD(s)\mathbb{H}(s)ds
  +\frac{\kappa_1}{4}\iint_{Q_t}\widetilde{\theta}_x^2dxds+C\iint_{Q_t}\widetilde{\theta}^2dxds\\
&\leq  C\int_0^tD(s)\mathbb{H}(s)ds +\frac{\kappa_1}{4}\iint_{Q_t}\widetilde{\theta}_x^2dxds.\\
\end{split}
\end{equation*}
By Young's inequality, we have
\begin{equation*}\label{u111}
\begin{split}
 & E_3+E_4 +E_7\\
 &\leq \epsilon\iint_{Q_t}  \widetilde{u}_x ^2dxds + \frac{C}{\epsilon}\int_0^t\big(1+\|u_x\|_{L^\infty(\Omega)}^2+\|\overline u_x\|_{L^\infty(\Omega)}^2\big)\int_{\Omega} \widetilde{\theta}^2dxds\\
 &\leq \epsilon\iint_{Q_t}  \widetilde{u}_x ^2dxds + \frac{C}{\epsilon}\int_0^tD(s)\mathbb{H}(s)ds,\quad \forall\epsilon \in (0, 1).
\end{split}
\end{equation*}
By the mean value theorem and $C^{-1}\leq \rho, \overline\rho, \theta, \overline\theta \leq C$, we obtain
$$
|\kappa(\rho,\theta)-\kappa(\overline\rho,\overline\theta)|\leq C(|\widetilde{\rho} |+|\widetilde{\theta}|),
$$
so
 \begin{equation*}\label{u9}
\begin{split}
 E_6 \leq & \frac{\kappa_1}{4}\iint_{Q_t} \widetilde{\theta}_x^2dxds+C\iint_{Q_t}|\overline\theta_x|^2\big(\widetilde{\rho}^2+\widetilde{\theta}^2\big)dxds \\   \leq & \frac{\kappa_1}{4} \iint_{Q_t} \widetilde{\theta}_x^2dxds+C\int_0^t\|\overline\theta_x\|^2_{L^\infty(\Omega)}\int_{\Omega}\big(\widetilde{\rho}^2+\widetilde{\theta}^2\big)dxds\\
    \leq &  \frac{\kappa_1}{4} \iint_{Q_t} \widetilde{\theta}_x^2dxds+C\int_0^tD(s)\mathbb{H}(s)ds.
\end{split}
\end{equation*}
By \eqref{wx4}, we have
\begin{equation*}\label{u10}
\begin{split}
 E_8\leq C\iint_{Q_t}\widetilde{\theta}^2dxds
 +C\mu^2\iint_{Q_t}  |\mathbf{w}_x|^4 dxds
 \leq  C\sqrt{\mu} + C\int_0^tD(s)\mathbb{H}(s)ds.
\end{split}
\end{equation*}
As to $E_9$, we have by the relation:
$|\mathbf{b}_x|^2-|\overline{\mathbf{b}}_x|^2=(\mathbf{b}_x+\overline{\mathbf{b}}_x)\cdot\widetilde{\mathbf{b}}_x$
\begin{equation*}
\begin{split}
 E_9\leq & \epsilon\iint_{Q_t}|\widetilde{\mathbf{b}}_x|^2dxds
 +\frac{C}{\epsilon}\int_0^t\big(\|\mathbf{b}_x\|_{L^\infty(\Omega)}^2+\|\overline{\mathbf{b}}_x\|_{L^\infty(\Omega)}^2\big)
 \int_{\Omega}\widetilde{\theta}^2dxds\\
 \leq & \epsilon\iint_{Q_t}|\widetilde{\mathbf{b}}_x|^2dxds
 +\frac{C}{\epsilon}\int_0^tD(s)\mathbb{H}(s)ds,~~\forall\epsilon\in(0, 1).
\end{split}
\end{equation*}
Substituting the results into \eqref{21} completes the proof to
\eqref{Theta}.

\indent{\bf Step 4} We claim that
\begin{equation}\label{0v3}
\begin{split}
  & \int_\Omega |\widetilde{\mathbf{w}}|^2dx   \leq C\sqrt{\mu} +\epsilon\iint_{Q_t} |\widetilde{\mathbf{b}}_x|^2 dxds + \frac{C}{\epsilon}\int_0^tD(s)\mathbb{H}(s)ds,\quad\forall \epsilon \in (0, 1).
  \end{split}
  \end{equation}
 From \eqref{e1}$_3$  and
  \eqref{equations}$_3$, we have
 \begin{equation*}
\begin{split}
& \rho\widetilde{\mathbf{w}}_t+ \rho u \widetilde{\mathbf{w}}_x+  \rho \widetilde{u} \overline{\mathbf{w}}_x-\widetilde{\mathbf{b}}_x+\frac{\widetilde{\rho}}{\overline\rho}\overline{\mathbf{b}}_x= \mu \mathbf{w}_{xx}.
  \end{split}
  \end{equation*}
  Multiplying it by $\widetilde{\mathbf{w}}$ and integrating over $Q_t$, we have
  \begin{equation}\label{0v33}
\begin{split}
   \frac12\int_\Omega \rho|\widetilde{\mathbf{w}}|^2dx
 & =\mu\iint_{Q_t}  \mathbf{w}_{xx}\cdot\widetilde{\mathbf{w}} dxds
  -\iint_{Q_t} \rho\widetilde{u}\overline{\mathbf{w}}_x\cdot\widetilde{\mathbf{w}} \\
  &\quad+ \iint_{Q_t}  \widetilde{\mathbf{b}}_x\cdot\widetilde{\mathbf{w}}dxds
-\iint_{Q_t}\frac{\widetilde{\rho}}{\overline\rho}
  \overline{\mathbf{b}}_x \cdot\widetilde{\mathbf{w}}dxds\\
  &\leq C\mu^2\iint_{Q_t} |\mathbf{w}_{xx}|^2dxds+C\iint_{Q_t} \widetilde{u}^2|\overline{\mathbf{w}}_x|^2dxds
 \\
 &\quad +\frac{C}{\epsilon}\int_0^tD(s)\mathbb{H}(s)ds  +\epsilon\iint_{Q_t} |\widetilde{\mathbf{b}}_x|^2 dxds,\quad\forall \epsilon\in(0, 1).
  \end{split}
  \end{equation}
 Observe that
   \begin{equation*}
\begin{split}
 & |\widetilde{u}(x,t)|= |u(x,t)-\overline u(x,t)|=\left|\int_0^x\widetilde{u}_x dx\right|\leq \left(\int_0^1\widetilde{u}_x^2dx\right)^{1/2}\omega^{1/2}(x),
 \forall x \in [0, 1/2],\\
 &|\widetilde{u}(x,t)|= |u(x,t)-\overline u(x,t)|=\left|\int_x^1\widetilde{u}_x dx\right|\leq \left(\int_0^1\widetilde{u}_x^2dx\right)^{1/2}\omega^{1/2}(x),
 \forall x \in [1/2, 1].\\
 \end{split}
  \end{equation*}
We have
    \begin{equation*}
\begin{split}
  |\widetilde{u}(x,t)|^2  \leq \left(\int_0^1\widetilde{u}_x^2dx\right) \omega(x),\quad \forall (x,t) \in \overline Q_T,
  \end{split}
  \end{equation*}
  which together with Lemma \ref{2.13} and \eqref{u14} gives
   \begin{equation*}
\begin{split}
 \iint_{Q_t} \widetilde{u}^2|\overline{\mathbf{w}}_x|^2dxds
  \leq & \int_0^t\left(\int_0^1\widetilde{u}_x^2dx\right) \left(\int_\Omega|\overline{\mathbf{w}}_x|^2\omega dx\right)ds
  \leq  C\iint_{Q_t}\widetilde{u}_x^2dxds\\
  \leq & C \int_0^tD(s)\mathbb{H}(s)ds.
  \end{split}
  \end{equation*}
Substituting it into \eqref{0v33} completes the proof to \eqref{0v3}.

\indent{\bf Step 5} We claim that
 \begin{equation}\label{B}
\begin{split}
  & \int_\Omega |\widetilde{\mathbf{b}}|^2dx +
   \iint_{Q_t} |\widetilde{\mathbf{b}}_x|^2dxds\leq C\int_0^t D(s)\mathbb{H}(s)ds.
  \end{split}
  \end{equation}
   From  \eqref{e1}$_4$  and \eqref{equations}$_4$, we have
   \begin{equation*}
\begin{split}
& \widetilde{\mathbf{b}}_t+  \big(u \widetilde{\mathbf{b}}\big)_x
+  \big(\widetilde{u}\overline{\mathbf{b}}\big)_x- \widetilde{\mathbf{w}}_x-\nu\widetilde{\mathbf{b}}_{xx}=0.
  \end{split}
  \end{equation*}
  Multiplying it by $\widetilde{\mathbf{b}}$ and integrating over $Q_t$ yield
  \begin{equation*}
\begin{split}
  &\frac12\int_\Omega |\widetilde{\mathbf{b}}|^2dx +
  \nu\iint_{Q_t} |\widetilde{\mathbf{b}}_x|^2dxds \\
 & = -\frac12\iint_{Q_t}u_x|\widetilde{\mathbf{b}}|^2dxds+\iint_{Q_t} \widetilde{u}  \overline{\mathbf{b}}\cdot\widetilde{\mathbf{b}}_xdxds- \iint_{Q_t} \widetilde{\mathbf{b}}_x \cdot\widetilde{\mathbf{w}}dxds\\
  &\leq  \frac{\nu}{2}\iint_{Q_t} |\widetilde{\mathbf{b}}_x|^2 dxds +C\int_0^t D(s)\mathbb{H}(s)ds.
  \end{split}
  \end{equation*}
  Thus, the  claim \eqref{B} is proved.

  Adding   the above five inequalities and taking a small $\epsilon>0$, we complete  the proof of Theorem 1.1(iii)
  by Gronwall's inequality.

Thus, the proof to Theorem 1.1 is  complete.

\section{Proof of Theorem 1.3}

\begin{lemma}\label{3.1}Let  \eqref{kappa}, \eqref{assumption1} and \eqref{vw} hold. Then $\overline{\mathbf{b}}=\overline{\mathbf{w}}=0$. Moreover,
\begin{equation*}
\begin{split}
&\sup\limits_{0<t<T}\int_\Omega \big(|\mathbf{b}|^2+|\mathbf{w}|^2\big) dx+\iint_{Q_T}|\mathbf{b}_x|^2dxdt \leq C\sqrt{\mu}.
\end{split}
\end{equation*}
 \end{lemma}
\begin{proof}From Theorem 1.1(iii), it suffices to show that $\overline{\mathbf{b}}=\overline{\mathbf{w}}=0$.
To this end, multiplying the equations \eqref{equations}$_3$ and
\eqref{equations}$_4$ by $\overline{\mathbf{w}}$ and
$\overline{\mathbf{b}}$, respectively, and integrating over $Q_t$,
we have
\begin{equation*}
\begin{split}
&\frac12\int_\Omega\overline\rho |\overline{\mathbf{w}}|^2dx-\iint_{Q_t} \overline{\mathbf{b}}_x\cdot\overline{\mathbf{w}}dxds=0,\\
 &\frac12\int_\Omega |\overline{\mathbf{b}}|^2 dx+\nu\iint_{Q_t}|\overline{\mathbf{b}}_x|^2dxds+\iint_{Q_t} \overline{\mathbf{b}}_x\cdot\overline{\mathbf{w}}dxds+\frac12\iint_{Q_t}\overline u_x |\overline{\mathbf{b}}|^2dxds=0.
\end{split}
\end{equation*}
 Adding the two equations yields
 \begin{equation*}
\begin{split}
&\frac12\int_\Omega
\big(\overline\rho|\overline{\mathbf{w}}|^2+|\overline{\mathbf{b}}|^2
\big)dx +\nu\iint_{Q_t}|\overline{\mathbf{b}}_x|^2dxdt\leq
\frac12\int_0^t\|\overline
u_x\|_{L^\infty(\Omega)}\int_\Omega|\overline{\mathbf{b}}|^2dxds,
\end{split}
\end{equation*}
which together with Gronwall's inequality completes the proof.

\end{proof}

\begin{lemma}\label{3.2}Let  \eqref{kappa}, \eqref{assumption1} and \eqref{vw} hold. Then
\begin{equation*}
\begin{split}
  \sup\limits_{0<t<T}\int_{\Omega}|\mathbf{w}_{x}|^2\omega^2dx  \leq C\sqrt{\mu}.\\
\end{split}
\end{equation*}
\end{lemma}
\begin{proof}
 Note that $\mathbf{w}_0=\mathbf{b}_0\equiv0$. By means of Lemmas \ref{2.13} and 3.1,  similar arguments as in Lemmas \ref{2.7} and \ref{2.8} give
\begin{equation*}
\begin{split}
    & \iint_{Q_t}|\mathbf{b}_{xx}|^2\omega^2dxds
    \leq C\sqrt{\mu}+C\iint_{Q_t}  |\mathbf{w}_{x}|^2\omega^2  dxds,\\[1mm]
    &\int_{\Omega}|\mathbf{w}_{x}|^2\omega^2dx
    \leq C\sqrt{\mu}     +C\iint_{Q_t}|\mathbf{b}_{xx}|^2\omega^2dxds,
  \end{split}
\end{equation*}
so
\begin{equation*}
\begin{split}
   \int_{\Omega}|\mathbf{w}_{x}|^2\omega^2dx
    \leq C\sqrt{\mu}     +C\iint_{Q_t}|\mathbf{w}_{x}|^2\omega^2dxds,
  \end{split}
\end{equation*}
and then, Gronwall's inequality yields the desired result.
The proof is complete.
\end{proof}

\begin{lemma}\label{3.3}Let  \eqref{kappa}, \eqref{assumption1} and \eqref{vw} hold. Then
\begin{equation*}
\begin{split}
  \iint_{Q_T}|\mathbf{w}_{t}|^2dxdt \leq C\sqrt{\mu}.\\
\end{split}
\end{equation*}
\end{lemma}
\begin{proof}
Using Lemmas \ref{2.10}, \ref{2.13}, \ref{3.1} and \ref{3.2} and noticing \eqref{u3}, we derive from  \eqref{w12} that
\begin{equation*}
\begin{split}
     \iint_{Q_T}|\mathbf{w}_{t}|^2dxdt &\leq C\mu^2\iint_{Q_T}|\mathbf{w}_{xx}|^2dxdt+C\iint_{Q_T}|\mathbf{b}_{x}|^2dxdt+C\iint_{Q_T}u^2|\mathbf{w}_{x}|^2dxdt\\
     &\leq C\sqrt{\mu}+\int_0^T\|u_x\|_{L^\infty(\Omega)}^2\int_\Omega|\mathbf{w}_{x}|^2\omega^2dxdt\leq C\sqrt{\mu}.\\
     \end{split}
\end{equation*}
The proof is complete.
\end{proof}

\begin{lemma}\label{3.4}Let  \eqref{kappa}, \eqref{assumption1} and \eqref{vw} hold. Then
\begin{equation*}
\begin{split}
 \sup\limits_{0<t<T}\int_{\Omega}|\mathbf{b}_{x}|^2 dx  +\iint_{Q_T}|\mathbf{b}_{t}|^2dxdt \leq C\sqrt{\mu}.\\
\end{split}
\end{equation*}
\end{lemma}
\begin{proof}
Multiplying \eqref{e1}$_4$ by $\mathbf{b}_t$,  integrating over
$Q_t$ and noticing $\mathbf{w}_0=0$, we have
\begin{equation}\label{b999}
\begin{split}
  &\frac{\nu}{2}\int_\Omega|\mathbf{b}_{x}|^2dx+ \iint_{Q_t} |\mathbf{b}_t|^2
  dxdt
  = \iint_{Q_t} \mathbf{w}_x\cdot\mathbf{b}_t dxdt  - \iint_{Q_t}(u\mathbf{b})_{x}\cdot\mathbf{b}_t dxdt.  \\
  \end{split}
\end{equation}
Using Lemmas \ref{2.10}, \ref{3.1} and \ref{3.3}, we obtain by the similar
arguments as in \eqref{b10} and \eqref{b11}
\begin{equation*}
\begin{split}
   \iint_{Q_t}   \mathbf{w}_x  \cdot\mathbf{b}_t dxdt
  &  \leq C\sqrt{\mu}+\frac{\nu}{4}\int_\Omega|\mathbf{b}_x|^2dx,
  \end{split}
\end{equation*}
and
\begin{equation*}
\begin{split}
   &-\iint_{Q_t}  (u\mathbf{b})_{x} \cdot\mathbf{b}_t dxdt \leq C\sqrt{\mu}+ \frac12  \iint_{Q_t} |\mathbf{b}_t|^2
   dxdt.
\end{split}\end{equation*}
Substituting them into \eqref{b999}, we complete the proof.
\end{proof}

Now we can prove Theorem 1.3.
 \vskip0.3cm
\noindent{\bf Proof of Theorem 1.3}~~By Theorem 1.1 (ii)-(iii) and Lemma \ref{3.1}, one sees that there exists a unique solution
$(\overline \rho,\overline u, \mathbf{0}, \mathbf{0}, \overline\theta)$ for
the limit problem \eqref{equations}   in $\mathbb{F}$.

Next, we are ready to show the second part of this theorem. Denote $\omega_\delta : [0,
1]\rightarrow [0, 1]$ for $\delta\in (0, 1/2)$ by
\begin{equation*}
\omega_\delta (x)=\left\{\begin{split}&x,&&0\leq x\leq \delta,\\
&\delta,&&\delta\leq x\leq 1-\delta,\\
&1-x,&&1-\delta\leq x \leq 1.\\
\end{split}\right.
\end{equation*}
Multiplying \eqref{w12} by $\mathbf{w}_{xx}\omega_\delta
^n(x)~(n=1,2,\cdots)$ and integrating over $Q_t$, we have
\begin{equation}\label{0v101}
\begin{split}
    \mu\iint_{Q_t}|\mathbf{w}_{xx}|^2\frac{ \omega_\delta ^n}{\rho}dxds
    &=\iint_{Q_t} \mathbf{w}_t\cdot \mathbf{w}_{xx} \omega_\delta ^n dxdt+\iint_{Q_t} u\mathbf{w}_x\cdot \mathbf{w}_{xx} \omega_\delta ^n dxds\\
    &\quad-\iint_{Q_t}  \mathbf{b}_x\cdot \mathbf{w}_{xx} \frac{ \omega_\delta ^n}{\rho} dxds.
  \end{split}
\end{equation}
Integrating by parts, using \eqref{w12} and noticing $\mathbf{w}_0=0$, we have
\begin{equation}\label{wtxx}
\begin{split}
    &\iint_{Q_t} \mathbf{w}_t\cdot \mathbf{w}_{xx} \omega_\delta ^n dxdt\\
   &=-\frac{1}{2}\int_\Omega|\mathbf{w}_x|^2 \omega_\delta ^n   dx
  -n\iint_{Q_t} \mathbf{w}_t\cdot \mathbf{w}_{x} \omega_\delta ^{n-1} \omega_\delta ' dxdt\\
   &=-\frac{1}{2}\int_\Omega|\mathbf{w}_x|^2 \omega_\delta ^n  dx
   -n\iint_{Q_t} \left(\frac{\mu}{\rho}\mathbf{w}_{xx}-u\mathbf{w}_x
   +\frac{\mathbf{b}_x}{\rho}\right)\cdot \mathbf{w}_{x} \omega_\delta ^{n-1}  \omega_\delta '  dxds\\
   &\leq  -\frac{1}{2}\int_\Omega|\mathbf{w}_x|^2 \omega_\delta ^n   dx+ \frac{\mu}{2}\iint_{Q_t}|\mathbf{w}_{xx}|^2\frac{ \omega_\delta ^n}{\rho}dxds +C_n\mu\iint_{Q_t}|\mathbf{w}_x|^2 \omega_\delta ^{n-2}
   dxds\\
   &\quad+C_n\iint_{Q_t}|u||\mathbf{w}_x|^2 \omega_\delta ^{n-1}|\omega_\delta '| dxds -n\iint_{Q_t}
    \mathbf{b}_x\cdot \mathbf{w}_{x}   \frac{\omega_\delta ^{n-1}  \omega_\delta '}{\rho} dxds,\quad n=2,3,\cdots.
  \end{split}
\end{equation}
Here and in what follows,     $C$ and $C_n$ are positive constants
independent of $\mu$ and $\delta$.

By the mean value theorem and $u(1,t)=u(0,t)=0$, we have
\begin{equation}\label{u33}
\begin{split}
|u(x,t)|\leq \|u_x\|_{L^\infty(\Omega)}\omega_\delta (x),\quad\forall x\in [0, \delta]\cup[1-\delta,1],\\
\end{split}
\end{equation}
which together with the definition of $\omega_\delta $ gives
\begin{equation*}
\begin{split}
    \iint_{Q_t}|u||\mathbf{w}_x|^2 \omega_\delta ^{n-1}|\omega_\delta '|  dxds=&\int_0^t\hspace{-2mm}\int_0^\delta|u||\mathbf{w}_x|^2 \omega_\delta ^{n-1}dxds+\int_0^t\hspace{-2mm}\int_{1-\delta}^1|u||\mathbf{w}_x|^2 \omega_\delta ^{n-1}dxds\\
    \leq& \int_0^t\|u_x\|_{L^\infty(\Omega)} \int_\Omega|\mathbf{w}_x|^2 \omega_\delta ^ndxds,\\
  \end{split}
\end{equation*}
thus
\begin{equation*}
\begin{split}
   & \iint_{Q_t} \mathbf{w}_t\cdot \mathbf{w}_{xx} \omega_\delta ^n dxdt\\
   &\leq -\frac{1}{2}\int_\Omega|\mathbf{w}_x|^2 \omega_\delta ^n   dx
   + \frac{\mu}{2}\iint_{Q_t}|\mathbf{w}_{xx}|^2\frac{\omega_\delta ^{n}}{\rho}dxds +C_n\mu\iint_{Q_t}|\mathbf{w}_x|^2 \omega_\delta ^{n-2}
   dxds\\
   &\quad+C\int_0^t \|u_x\|_{L^\infty(\Omega)} \int_\Omega|\mathbf{w}_x|^2 \omega_\delta ^ndxds-n\iint_{Q_t}
     \mathbf{b}_x\cdot \mathbf{w}_{x} \frac{\omega_\delta ^{n-1}  \omega_\delta '}{\rho}  dxds.\\
  \end{split}
\end{equation*}
To estimate the second integral on the right-hand side of
\eqref{0v101}, we have by integrating by parts  and noticing
\eqref{u33}
\begin{equation*}
\begin{split}
  \iint_{Q_t} u\mathbf{w}_x\cdot \mathbf{w}_{xx} \omega_\delta ^n dxds
  =&-\frac{1}{2}\iint_{Q_t} |\mathbf{w}_x|^2[u_x \omega_\delta ^n+nu \omega_\delta ^{n-1} \omega_\delta ']
  dxds\\
  \leq&C_n\int_0^t \|u_x\|_{L^\infty(\Omega)} \int_\Omega|\mathbf{w}_x|^2 \omega_\delta ^ndxds.\\
  \end{split}
\end{equation*}
As to the third term on the right-hand side of \eqref{0v101}, we
have
\begin{equation*}
\begin{split}
   &-\iint_{Q_t} \mathbf{b}_x\cdot \mathbf{w}_{xx}\frac{ \omega_\delta ^n}{\rho} dxds\\[1mm]
   &=\iint_{Q_t} \mathbf{w}_x\cdot \mathbf{b}_{xx} \frac{ \omega_\delta ^n}{\rho}dxds-\iint_{Q_t} \mathbf{w}_x\cdot \mathbf{b}_{x} \frac{\omega_\delta ^n\rho_x}{\rho^2} dxds+n\iint_{Q_t} \mathbf{b}_{x}\cdot\mathbf{w}_x\frac{ \omega_\delta ^{n-1} \omega_\delta '}{\rho} dxds\\[1mm]
& \leq C\iint_{Q_t}  |\mathbf{b}_{xx}|^2 \omega_\delta ^n dxds+C\iint_{Q_t}  |\mathbf{w}_x|^2 \omega_\delta ^n dxds+C\iint_{Q_t}  |\mathbf{b}_x|^2 \omega_\delta ^n \rho_x^2dxds\\[1mm]
&\quad+n\iint_{Q_t} \mathbf{b}_{x}\cdot\mathbf{w}_x\frac{ \omega_\delta ^{n-1} \omega_\delta '}{\rho} dxds\\[1mm]
&\leq C\sqrt{\mu}\delta^{n-1}+C\iint_{Q_t}  |\mathbf{b}_{xx}|^2 \omega_\delta ^n dxds + C\iint_{Q_t}  |\mathbf{w}_x|^2 \omega_\delta ^n dxds\\[1mm]
&\quad+n\iint_{Q_t} \mathbf{b}_{x}\cdot\mathbf{w}_x\frac{ \omega_\delta ^{n-1} \omega_\delta
'}{\rho} dxds,
  \end{split}
\end{equation*}
where we  used the fact  by \eqref{rho}$_1$, Lemma \ref{3.1} and $0\leq \omega_\delta
(x)\leq \delta$
\begin{equation*}
\begin{split}
     \iint_{Q_t}  |\mathbf{b}_x|^2 \omega_\delta ^n \rho_x^2dxds\leq&C\int_0^t  \left\||\mathbf{b}_x|^2 \omega_\delta ^n\right\|_{L^\infty(\Omega)} ds
     \leq  C \iint_{Q_t}  \left|(|\mathbf{b}_x|^2 \omega_\delta ^n)_x\right| dxds\\[1mm]
    \leq & C_n\iint_{Q_t}    |\mathbf{b}_{x}|^2| \omega_\delta ^{n-1} \omega_\delta '|dxds+C \iint_{Q_t}    |\mathbf{b}_x\cdot\mathbf{b}_{xx}|  \omega_\delta ^n dxds\\[1mm]
    \leq &C_n\sqrt{\mu}\delta^{n-1} +C\iint_{Q_t}    |\mathbf{b}_{xx}|^2  \omega_\delta ^n dxds.
  \end{split}
\end{equation*}
Substituting the above results into \eqref{0v101} yields
\begin{equation}\label{w000}
\begin{split}
     &\int_\Omega|\mathbf{w}_x|^2 \omega_\delta ^n   dx+\mu \iint_{Q_t} |\mathbf{w}_{xx}|^2 \omega_\delta ^n   dxds\\
     & \leq C_n\sqrt{\mu}\delta^{n-1}+C_n\mu\iint_{Q_t}|\mathbf{w}_x|^2 \omega_\delta ^{n-2}
   dxds \\
    &\quad+ C \int_0^t \big[1+\|u_x\|_{L^\infty(\Omega)}\big] \int_\Omega|\mathbf{w}_x|^2 \omega_\delta ^n  dxds+C\iint_{Q_t} |\mathbf{b}_{xx}|^2 \omega_\delta ^n   dxds.
  \end{split}
\end{equation}
It remains to treat the relation between the terms $\iint_{Q_t}
|\mathbf{b}_{xx}|^2 \omega_\delta ^n dxds$ and  $\iint_{Q_t}
|\mathbf{w}_{x}|^2 \omega_\delta ^n dxds$. To this end, we multiply
\eqref{e1}$_4$ by $\mathbf{b}_{xx} \omega_\delta ^n(x) ~(n=2,3,\cdots)$ and integrate over $Q_t$ to obtain
\begin{equation}\label{b1111}
\begin{split}
    \nu\iint_{Q_t} |\mathbf{b}_{xx}|^2 \omega_\delta ^n dxds
    &=\iint_{Q_t} \mathbf{b}_t\cdot \mathbf{b}_{xx} \omega_\delta ^n dxdt+\iint_{Q_t}  (u\mathbf{b})_x\cdot \mathbf{b}_{xx} \omega_\delta ^n  dxds\\
    &\quad-\iint_{Q_t} \mathbf{w}_x\cdot \mathbf{b}_{xx} \omega_\delta ^n dxds.
  \end{split}
\end{equation}
To estimate the first term on right-hand side of \eqref{b1111}, we
use Young's inequality, Lemma  \ref{3.4} and $0\leq \omega_\delta
(x)\leq \delta$  to obtain
\begin{equation*}
\begin{split}
   \iint_{Q_t} \mathbf{b}_t\cdot \mathbf{b}_{xx} \omega_\delta ^n dxdt
   \leq C\sqrt{\mu}\delta^{n}+\frac{\nu}{4}\iint_{Q_t}|\mathbf{b}_{xx}|^2 \omega_\delta ^n dxdt.
  \end{split}
\end{equation*}
Next we deal with the second term on right-hand side of
\eqref{b1111}. By $0\leq \omega_\delta (x)\leq \delta$ and Lemmas
\ref{2.10} and \ref{3.1}, we have
\begin{equation*}
\begin{split}
    \iint_{Q_t}  |(u\mathbf{b})_x|^2 \omega_\delta ^n  dxds \leq &C\iint_{Q_t} u^2|\mathbf{b}_{x}|^2 \omega_\delta ^n  dxds+C\iint_{Q_t} u_x^2|\mathbf{b} |^2 \omega_\delta ^n  dxds\\
  \leq & C\sqrt{\mu}\delta^n+C\delta^n\int_0^t\|u_x\|_{L^\infty(\Omega)}^2\int_\Omega|\mathbf{b} |^2  dxds\leq C\sqrt{\mu}\delta^n.
  \end{split}
\end{equation*}
As to the third term on right-hand side of \eqref{b1111}, we have by Young's inequality
\begin{equation*}
\begin{split}
   - \iint_{Q_t} \mathbf{w}_x\cdot \mathbf{b}_{xx} \omega_\delta ^n dxds\leq C\iint_{Q_t} |\mathbf{w}_x|\omega_\delta ^n dxds
 +\frac{\nu}{4}\iint_{Q_t} |\mathbf{b}_{xx}|\omega_\delta ^n dxds.
 \end{split}
\end{equation*}
Substituting them into \eqref{b1111} yields
\begin{equation}\label{bel}
\begin{split}
     & \iint_{Q_t} |\mathbf{b}_{xx}|^2 \omega_\delta ^n   dxds\leq C\sqrt{\mu}\delta^{n}+C \iint_{Q_t}|\mathbf{w}_x|^2 \omega_\delta ^{n}dxds,
  \end{split}
\end{equation}
and   inserting it into \eqref{w000}  and  using Gronwall's
inequality, we obtain the iteration
\begin{equation}\label{iteration}
\begin{split}
     &\int_\Omega|\mathbf{w}_x|^2 \omega_\delta ^n   dx  \leq C_n \sqrt{\mu}\delta^{n-1}  + C_n\mu\iint_{Q_t} |\mathbf{w}_x|^2 \omega_\delta ^{n-2}   dxds,\quad n=2,3,\cdots.
  \end{split}
\end{equation}
Note that  the above results still hold for $n=1$. Since the term
$-\iint_{Q_t} \frac{\mu}{\rho}\mathbf{w}_{xx}\cdot \mathbf{w}_{x}
\omega_\delta ' dxds$ in the equality of \eqref{wtxx} with $n=1$ can
be dealt with as follows
\begin{equation*}
\begin{split}
     &-\iint_{Q_t} \frac{\mu}{\rho}\mathbf{w}_{xx}\cdot \mathbf{w}_{x} \omega_\delta ' dxds\leq C\sqrt{\mu}\iint_{Q_t} |\mathbf{w}_{x}|^2 dxds+C\mu^{3/2}\iint_{Q_t}|\mathbf{w}_{xx}|^2dxds\leq C,
        \end{split}
\end{equation*}
where we used Lemma \ref{2.13},  a similar argument as above gives,
instead of \eqref{iteration},
\begin{equation}\label{wx3}
\begin{split}
     &\int_\Omega|\mathbf{w}_x|^2 \omega_\delta     dx  \leq C.
  \end{split}
\end{equation}
So, we derive from \eqref{iteration} that
\begin{equation*}
\begin{split}
     &\int_\Omega|\mathbf{w}_x|^2 \omega_\delta ^2   dx  \leq  C(\sqrt{\mu}\delta+\sqrt{\mu}),\\
     &\int_\Omega|\mathbf{w}_x|^2 \omega_\delta ^3   dx  \leq C\big(\sqrt{\mu}\delta^2+\mu\big).\\
        \end{split}
\end{equation*}
Next, taking $n=4,5,6,7$ in
\eqref{iteration}, respectively, we get
\begin{equation*}
\begin{split}
            &\int_\Omega|\mathbf{w}_x|^2 \omega_\delta ^4   dx  \leq C\big(\sqrt{\mu}\delta^3+\mu^{3/2}\delta+\mu^{3/2}\big),\\
      &\int_\Omega|\mathbf{w}_x|^2 \omega_\delta ^5   dx  \leq C\big(\sqrt{\mu}\delta^4+\mu^{3/2}\delta^2+\mu^{2}\big),\\
        &\int_\Omega|\mathbf{w}_x|^2 \omega_\delta ^6   dx  \leq C\big(\sqrt{\mu}\delta^5+\mu^{3/2}\delta^3+\mu^{5/2}\delta+\mu^{5/2}\big),\\
         &\int_\Omega|\mathbf{w}_x|^2 \omega_\delta ^7  dx
         \leq C\big(\sqrt{\mu}\delta^6+\mu^{3/2}\delta^4+\mu^{5/2}\delta^2+\mu^3\big),\\
           \end{split}
\end{equation*}
thus,    an induction gives
\begin{equation}\label{0w1}
\begin{split}
      \int_\Omega|\mathbf{w}_x|^2 \omega_\delta ^n dx\leq \left\{
              \begin{split}
           & C_n\big(\sqrt{\mu}\delta^{n-1}+\mu^{3/2}\delta^{n-3} +\cdots+\mu^{(n-2)/2}\delta^{2}+ \mu^{(n-1)/2}\big)~~(n=\hbox{\rm odd}),\\
       & C_n\big(\sqrt{\mu}\delta^{n-1}+\mu^{3/2}\delta^{n-3} +\cdots+\mu^{(n-1)/2}\delta
       +\mu^{(n-1)/2}\big)~~(n=\hbox{\rm even}),\\
         \end{split}\right.
  \end{split}
\end{equation}
where $n=2,3,\cdots,$ which together with the definition of
$\omega_\delta $ gives
\begin{equation}\label{0w0}
\begin{split}
      \int_{\delta}^{1-\delta}|\mathbf{w}_x|^2 dx\leq \left\{
              \begin{split}
           & C_n\big(\tau+\tau^3+\cdots+\tau^{n-2}\big)+C_n\mu^{(n-1)/2}/\delta^n~(n=\hbox{\rm odd}),\\
       & C_n\big(\tau+\tau^3+\cdots+\tau^{n-1}\big)+C_n\mu^{(n-1)/2}/\delta^n~(n=\hbox{\rm even}),\\
         \end{split}\right.
  \end{split}
\end{equation}
where  $\delta\in (0, 1/2), \tau=\sqrt{\mu}/\delta.$

 On the other hand,  we have by the mean value theorem and
Lemma 3.1
\begin{equation*}
\begin{split}
     \|\mathbf{w}\|_{L^\infty(\delta, 1-\delta)} \leq & \frac{1}{1-2\delta}\int_\delta^{1-\delta}|\mathbf{w}|dx+\int_\delta^{1-\delta}|\mathbf{w}_x|dx\\
        \leq &C\mu^{1/4}+\left(\int_\delta^{1-\delta}|\mathbf{w}_x|^2dx\right)^{1/2},\quad \forall\delta\in(0, 1/4),
     \end{split}
\end{equation*}
which together with \eqref{0w0} implies that  any function
$\delta(\mu)$ with $\delta(\mu)\downarrow0$ and
$\frac{\mu^{(n-1)/(2n)}}{\delta(\mu)}=\frac{\mu^{1/2-1/(2n)}}{\delta(\mu)}\rightarrow 0$ as $ \mu
\rightarrow0$ satisfies
\begin{equation}\label{v99}
\begin{split}
   &\lim\limits_{\mu\rightarrow 0} \|\mathbf{w}\|_{L^\infty(0, T;L^\infty(\delta(\mu), 1-\delta(\mu)))}=0.
  \end{split}
\end{equation}
Since $n$ can be arbitrarily large, we see that $\delta(\mu)=  \mu^\alpha $  satisfies \eqref{v99}  for any $\alpha\in (0, 1/2)$.   The proof  is  completed.

\section*{Acknowledgments}
The authors would like to thank Professor Tong Yang at City
University of Hong Kong for valuable discussions on Lemma \ref{2.14} during their visit to City University of Hong Kong.
 The research was supported  in part by the NSFC (grants 11571062,11571380,11401078), Guangdong Natural Science
Foundation (grant 2014A030313161), the Program for Liaoning
Excellent Talents in University (grant LJQ2013124)
  and  the Fundamental Research Fund  for the Central Universities  (grant DC201502050202).
\par

\end{document}